\RequirePackage[l2tabu,orthodox,abort]{nag}

\documentclass{birkjour}

\usepackage{fixltx2e}
\usepackage[all,error]{onlyamsmath}

\usepackage{amssymb,amsfonts,amsmath,
amsrefs,enumerate,color,url,hyperref,
mathbbol,dsfont}
\usepackage{mathrsfs}
\usepackage{tikz}

\usepackage[T1]{fontenc}
\usepackage[cp1250]{inputenc}

\usepackage[strict=true]{csquotes}

 \newtheorem{thm}{Theorem}[section]
 \newtheorem{cor}[thm]{Corollary}
 \newtheorem{lem}[thm]{Lemma}
 \newtheorem{prop}[thm]{Proposition}
 \theoremstyle{definition}
 \newtheorem{defn}[thm]{Definition}
 \theoremstyle{remark}
 \newtheorem{rem}{Remark}
 \newtheorem{ex}{Example}
 \newtheorem*{exs*}{Examples}
 \numberwithin{equation}{section}
\usepackage{mathtools}

\DeclareMathAlphabet{\mathpzc}{OT1}{pzc}{m}{it}

\newcommand{\ka}{{\mathpzc k}}
\newcommand{\kad}{{\mathpzc K}}

\newcommand{\isuma}{\dot{\sum_{i\in \mc N}}}
\newcommand{\isumap}{\dot{\sum_{i\in \mc M}}}
\newcommand{\rlac}{\rla^{\textnormal{\hspace{0.02cm}\tiny co}}}
\newcommand{\rlace}{R_{\lam,\eps}^{\textnormal{\hspace{0.02cm}\tiny co}}}
\newcommand{\rmic}{\rmi^{\textnormal{\hspace{0.02cm}\tiny co}}}
\newcommand{\rladu}{\rla^{\textnormal{\hspace{0.02cm}\tiny du}}}
\newcommand{\du}{\textnormal{\hspace{0.02cm}\tiny du}}
\newcommand{\co}{\textnormal{\hspace{0.02cm}\tiny co}}
\newcommand{\rladue}{R_{\lam,\eps}^{\textnormal{\hspace{0.02cm}\tiny du}}}
\newcommand{\rmidu}{\rmi^{\textnormal{\hspace{0.02cm}\tiny du}}}
\newcommand{\aco}{A^{\textnormal{\hspace{0.02cm}\tiny co}}}
\newcommand{\macoe}{\mathcal A_\eps ^{\textnormal{\hspace{0.02cm}\tiny co}}}

\newcommand{\suu}{S_{\textnormal{\hspace{0.02cm}\tiny u}}}
\DeclareMathOperator{\Ker}{Ker}
\DeclareMathOperator{\wyr}{Det}

\usepackage{enumitem}
\renewcommand{\theenumi}{\textup{(\alph{enumi})}}
\renewcommand{\labelenumi}{\theenumi}

\begin{document}

\title[Concatenation of Feller processes]{Concatenation of dishonest Feller processes, exit laws, and limit theorems on graphs.} 

\author[A. Bobrowski]{Adam Bobrowski}

\address{
Lublin University of Technology\\
Nadbystrzycka 38A\\
20-618 Lublin, Poland}

\email{a.bobrowski@pollub.pl}








\newcommand{\cxi}{(\xi_i)_{i\in \N} }
\newcommand{\lam}{\lambda}
\newcommand{\eps}{\varepsilon}
\newcommand{\ud}{\, \mathrm{d}\hspace{0.02cm}}
\newcommand{\mud}{\mathrm{d}}
\newcommand{\pr}{\mathbb{P}}
\newcommand{\f}{\mathcal{F}}
\newcommand{\s}{\mathcal{S}}
\newcommand{\h}{\mathcal{H}}
\newcommand{\ai}{\mathcal{I}}
\newcommand{\R}{\mathbb{R}}
\newcommand{\C}{\mathbb{C}}
\newcommand{\Z}{\mathbb{Z}}
\newcommand{\N}{\mathbb{N}}
\newcommand{\Y}{\mathbb{Y}}
\newcommand{\e}{\mathrm {e}}
\newcommand{\tif}{\widetilde {f}}
\newcommand{\Id}{{\mathrm{Id}}}
\newcommand{\cic}{C_{\mathrm{mp}}}
\newcommand{\cod}{C_{\mathrm{odd}}[0,1]}
\newcommand{\cev}{C_{\mathrm{even}}[0,1]}
\newcommand{\cevr}{C_{\mathrm{even}}(\mathbb{R})}
\newcommand{\codr}{C_{\mathrm{odd}}(\mathbb{R})}
\newcommand{\cez}{C_0(0,1]}
\newcommand{\fod}{f_{\mathrm{odd}}} 
\newcommand{\fev}{f_{\mathrm{even}}} 
\newcommand{\sem}[1]{\mbox{$\left (\e^{t{#1}}\right )_{t \ge 0}$}}
\newcommand{\semi}[1]{\mbox{$\left ({#1}\right )_{t > 0}$}}
\newcommand{\semt}[2]{\mbox{$\left (\e^{t{#1}} \otimes_\varepsilon \e^{t{#2}} \right )_{t \ge 0}$}}
\newcommand{\tr}{\textcolor{red}}
\newcommand{\cea}{C_A}
\newcommand{\ceat}{C_A(t)}
\newcommand{\cosinea}{(\ceat )_{t\in \R}}  
\newcommand{\sea}{S_A}
\newcommand{\seat}{S_A(t)}
\newcommand{\sema}{(\seat )_{t\ge 0}}
\newcommand{\wt}{\widetilde}
\renewcommand{\iff}{if and only if }
\renewcommand{\k}{\mathrm{k}}
\newcommand{\tcm}{\textcolor{magenta}}
\newcommand{\ecm}{\textcolor{olive}}
\newcommand{\tcb}{\textcolor{blue}}
\newcommand{\dx}{\ \textrm {d} x}
\newcommand{\dy}{\ \textrm {d} y}
\newcommand{\dz}{\ \textrm {d} z}
\newcommand{\di}{\textrm{d}}
\newcommand{\tcg}{\textcolor{green}}
\newcommand{\lc}{\mathfrak L_c}
\newcommand{\ls}{\mathfrak L_s}
\newcommand{\grat}{\lim_{t\to \infty}}
\newcommand{\grar}{\lim_{r\to 1-}}
\newcommand{\graR}{\lim_{R\to 1+}}
\newcommand{\grak}{\lim_{\kappa \to \infty}}
\newcommand{\gra}{\lim_{n\to \infty}}
\newcommand{\grae}{\lim_{\eps \to 0}}
\newcommand{\rez}[1]{\left (\lam - #1\right)^{-1}}
\newcommand{\papa}{\hfill $\square$}
\newcommand{\papap}{\end{proof}}
\newcommand {\x}{\mathbb{X}}
\newcommand{\aex}{A_{\mathrm ex}}
\newcommand{\jcg}[1]{\left ( #1 \right )_{n\ge 1} }
\newcommand {\y}{\mathbb{Y}}
\newcommand{\injtp}{\x \hat \otimes_{\varepsilon} \y}
\newcommand{\pin}{\|_{\varepsilon}}
\newcommand{\mc}{\mathcal}
\newcommand{\inter}{\left [0, 1\right ]}
\newcommand{\lir}{\lim_{r \to 1}}
\newcommand{\ha}{\mathfrak {H}}
\newcommand{\dom}[1]{D(#1)}
\newcommand{\mquad}[1]{\quad\text{#1}\quad}
\newcommand{\lil}{\lim_{\lam \to \infty}}
\newcommand{\lilz}{\lim_{\lam \to 0+}}
\newcommand{\della}{\mathcal L_\lam^A} 
\newcommand{\dell}{\mathcal L_\lam}
\newcommand{\mem}{\mathfrak {p}}
\newcommand{\sdlam}{\sqrt{2\lam}}
\newcommand{\slam}{\sdlam} 
\newcommand{\ela}{e_\lam} 
\makeatletter
\newcommand{\normt}{\@ifstar\@normts\@normt}
\newcommand{\@normts}[1]{%
  \left|\mkern-1.5mu\left|\mkern-1.5mu\left|
   #1
  \right|\mkern-1.5mu\right|\mkern-1.5mu\right|
}
\newcommand{\@normt}[2][]{%
  \mathopen{#1|\mkern-1.5mu#1|\mkern-1.5mu#1|}
  #2
  \mathclose{#1|\mkern-1.5mu#1|\mkern-1.5mu#1|}
}
\makeatother

\thanks{Version of \today}

\begin{abstract}We provide a rather explicit formula for the resolvent of a~concatenation of $N$ processes in terms of their exit laws and certain probability measures characterizing the way the processes are concatenated. As an application, we prove an averaging principle saying that by concatenating asymptotically splittable processes one can approximate Markov chains. \end{abstract}

\maketitle

\newcommand{\oper}{\mathfrak R_r}
\newcommand{\opern}{\mathfrak R_{\rn}^\mho}
\newcommand{\brn}{\mbox{$\Delta^\mho_{\rn}$}}
\newcommand{\bro}{\mbox{$\Delta_{\rn}$}}
\newcommand{\rn}{r}
\newcommand{\cern}{C\hspace{-0.07cm}\left[\rn, 1\right ]}
\newcommand{\cernbez}{C\left[\rn, 1\right ]}
\newcommand{\cep}{C\hspace{-0.07cm}\left[ 0, 1\right ]}
\newcommand{\copi}{C[0,\pi]}
\newcommand{\CP}{\mbox{$C_p[0,2\pi]$}}
\newcommand{\cerec}{C\hspace{-0.07cm} \left ([0,\pi]\times [r,1]\right)}
\newcommand{\cerecbez}{C \left ([0,\pi]\times [r,1]\right)}
\newcommand{\cerecdwa}{C^2\hspace{-0.07cm} \left ([0,\pi]\times [r,1]\right)}
\newcommand{\cerecj}{C\hspace{-0.07cm} \left ([0,\pi]\times \left [0 ,1\right ]\right)}
\newcommand{\xprim}{C_\theta (UR)}
\newcommand{\ie}{i.e., }
\newcommand{\rla}{R_\lambda}
\newcommand{\grubex}{\mathbb X}
\newcommand{\Jcg}[1]{\left ( #1 \right )_{i=1,\dots,N}} 
\newcommand{\LB}{\mbox{$\Delta_{\text{{\tiny\textsf LB}}}$}}
\newcommand{\rod}[1]{\mbox{$\left (#1 \right )_{t\ge 0}$}} 

\section{Introduction} 
\newcommand{\cosa}{\mbox{$C(S_A)$}}
\newcommand{\cosb}{\mbox{$C(S_B)$}}
\newcommand{\coss}{\mbox{$C(S)$}}
\newcommand{\cosu}{\mbox{$C(\suu)$}}
\newcommand{\cosf}{\mbox{$C^\flat (S)$}}

\newcommand{\cosi}{\mbox{$C(S_i)$}}
\newcommand{\rlaa}{\rla^A}
\newcommand{\rlab}{\rla^B}
\newcommand{\rlag}{\rla^G}
\newcommand{\rlai}{R_{\lam,i}}
\newcommand{\rlamc }{\rla^{\mc A}}
\newcommand{\rmi}{R_\mu}  
\newcommand{\rmii}{R_{\mu,i}}
\subsection{Not honest stochastic processes as building blocks of more complex systems} 

In modeling natural phenomena one often encounters stochastic processes that are not honest, that is, are undefined after a random time. Such processes include e.g. so-called explosive Markov chains (and particular cases of birth and death processes among them) \cites{fmc,chungboundary,fellerboundaryforMC,norris}, Feller processes killed upon exiting a certain region \cite{chungodkb}, processes generated by fractional Laplace operators perturbed by gradients (see \cite{bogdanjakub}*{Eq. (20)}), and killed L\'evy processes (\cite{bogdanki}*{Section 5.5} and \cite{bogdanros}*{Section 4.5}).

In biology, they abound for example in the theory of coagulation and fragmentation processes  \cite{banaarlo,bll}, and form building blocks of models involving semi-permeable membranes, such as direct or facilitated diffusion 
of ions and molecules across cell/plasma membranes \cite{steinbook,khan}. To wit, a randomly moving particle that passes through a membrane can be seen as disappearing  from one part of the state-space and reappearing in its other part, and it is natural to see the entire process as built from simpler, not honest ones, defined only in parts of the state-space --- see below for more on this subject. 
Such a structure of the studied process is visible e.g. in the papers \cite{bobmor} and \cite{zmarkusem} (see also \cite{knigaz}) devoted to  reconciliation of two seemingly different models of so-called fast neurotransmitters, to describing activity of kinases, and to intracellular calcium dynamics.  It is also apparent, to name a few, (a) in  the work of J. E. Tanner \cite{tanner}, who studied diffusion of particles through a sequence of permeable barriers (see also Powles et al. \cite{powles}, for a continuation of the subject), (b) in the paper by S. S. Andrews \cite{andrews}, devoted to describing absorption and desorption phenomena, and (c) in Fireman et al. \cite{fireman}, where a compartment model with permeable walls (representing e.g., cells, and axons in the white matter of the brain in particular) is analyzed. 
See also the references given   by A. Lejay in \cite{lejayn} to articles devoted to modeling  of flows between cells, and the papers with models relevant for astrophysics, ecology, homogenization, geophysics and finance, provided in \cite{lejayskew}*{p. 414}.

In fact, it is the abundance of non-honest processes in the natural sciences that is a primary motivation of this work.  
As noted e.g. in Hartwell et al. \cite{hartwell} (see also \cite{alon}), biological systems are built of myriads of interacting components which can in themselves be decomposed into yet smaller subsystems with very specific interactions. 
The idea of \emph{concatenation} of not honest processes, that is, of building more complex structures from simpler ones, provides a unifying language for a wide range of applied models, and we believe that our main formula (see \eqref{concp}) is an efficient tool for analyzing them.

\subsection{Non-honest Feller processes and boundary conditions}


A number of examples of not honest processes can be found in the theory of Feller processes on domains with boundaries --- their behavior at these boundaries is customarily described by means of boundary conditions. For a simple concrete case, consider a standard Brownian motion on the right half-axis $\R^+\coloneqq [0,\infty)$ which after reaching the boundary point $x=0$ remains there for an exponential time with parameter, say, $a \ge 0$, and is undefined later, when --- figuratively speaking --- the Brownian traveller disappears from the state-space. This process, that will be referred to as \emph{elementary exit} (comp. \cite{fellera1}*{p. 3} or \cite{knigaz}*{p. 19}) is related to the boundary condition
\begin{equation}\label{intron:1} af''(0) +  f(0)=0. \end{equation}
The Robin boundary condition 
\begin{equation}\label{intron:2} bf'(0) = f(0), \end{equation}
where $b> 0$ describes a similar process, termed \emph{elastic Brownian motion}: initially it behaves precisely as the reflected Brownian motion but the time it spends `at the boundary' where $x=0$ is measured with the help of the celebrated L\'evy local time; when an exponential time with parameter $b^{-1}$ with respect to the L\'evy local time elapses, the process is no longer defined \cite{itop,ito,karatzas}.

The role of the \emph{sticky}, or \emph{slowly reflecting} boundary 
\begin{equation}\label{intron:3} af''(0) - bf'(0)+ cf(0)=0, \end{equation}
is similar --- see \cite{kostrykin2010brownian}, \cite{liggett}*{p. 127} or \cite{revuz}*{p. 421} --- the only difference lies in the fact that here the process has the tendency to stick to the boundary for a longer time than in the reflected Brownian motion.

\subsection{Concatenating processes that are not honest; transmission conditions}
In describing such processes it is one of the tricks of the trade to make them honest by adjoining an additional point (sometimes called the cemetery or coffin state) to the state-space and agreeing that from the moment when the process is undefined, it actually stays at this new point forever \cites{kniga,bass,kallenbergnew,fmc}. There are, however, other, more engaging ways the process may be {continued}, that is, extended to a possibly honest process. Interestingly, such a continuation can often be obtained by appropriate manipulations on boundary conditions. 

For example, in the case of elementary jump from the boundary at $x=0$, we can require that when the exponential time spent at $x=0$ elapses the process should start anew at a point $s $ of a locally compact space $S$, and evolve according to the rules governing there. To this end, one imposes the transmission condition
\begin{equation}\label{intron:4} f'' (0) = a [f(s) - f(0)]. \end{equation}
A similar modification  of the elastic Brownian motion results from changing the boundary condition \eqref{intron:2} to 
the transmission condition \begin{equation}\label{intron:5} bf' (0) = f(0) - f(s), \end{equation}
see \cites{bobmor,nagrafach,lejayn}, consult also \cite{knigaz}*{p. 66} and \cite{zmarkusem}*{p. 669}, where further references are given. 

For a more complex example we refer to \cite{arkuk} (see also \cite{arkuk2}), where a boundary condition of the form 
\[ f(x) = \int_\Omega f \ud \mu_x \] 
commands the process that leaves the state-space $\Omega$ at a point $x$ of its boundary to start anew inside the state-space at a randomly chosen point with distribution $\mu_x$. 

Although processes related to transmission conditions discussed above are motivating examples for our paper, the idea that the process which is no longer defined in its original state-space could be continued in an extended state-space, applies, of course, to more general situations than these described above, and to processes that may have nothing to do with transmission or boundary conditions. In fact, such constructions are known under the name of \emph{concatenation of processes}. 

In Section II.14 of Sharpe's monograph \cite{sharpe} (pp. 77-84) two processes, say, $X^1$ and $X^2$, the first of which is not honest, are concatenated as follows. A particle starting in the state-space of $X^1$ moves about according to the law of evolution of $X^1$, but when the lifetime of this process is over, the particle jumps to the state-space of $X^2$ and starts to move according to the law of $X^2$. A particular example of such concatenation is described in the classical treatise of Ito and McKean \cite{ito}*{p. 105} in intuitive probabilistic terms and by means of transmission conditions. 

In the recent paper of F. Werner \cite{werner} a more general construction is provided: given a sequence of Markov processes $X^n, n \ge 1 $ their concatenation is defined as follows: the new process starts at the state-space of $X^1$, behaves like this process for its lifetime, and then starts anew as $X^2$; when the lifetime of $X^2$ is over, the process starts as $X^3$, and so on, until it possibly reaches its coffin state. 
\begin{center}
\begin{figure}
\includegraphics[scale=1]{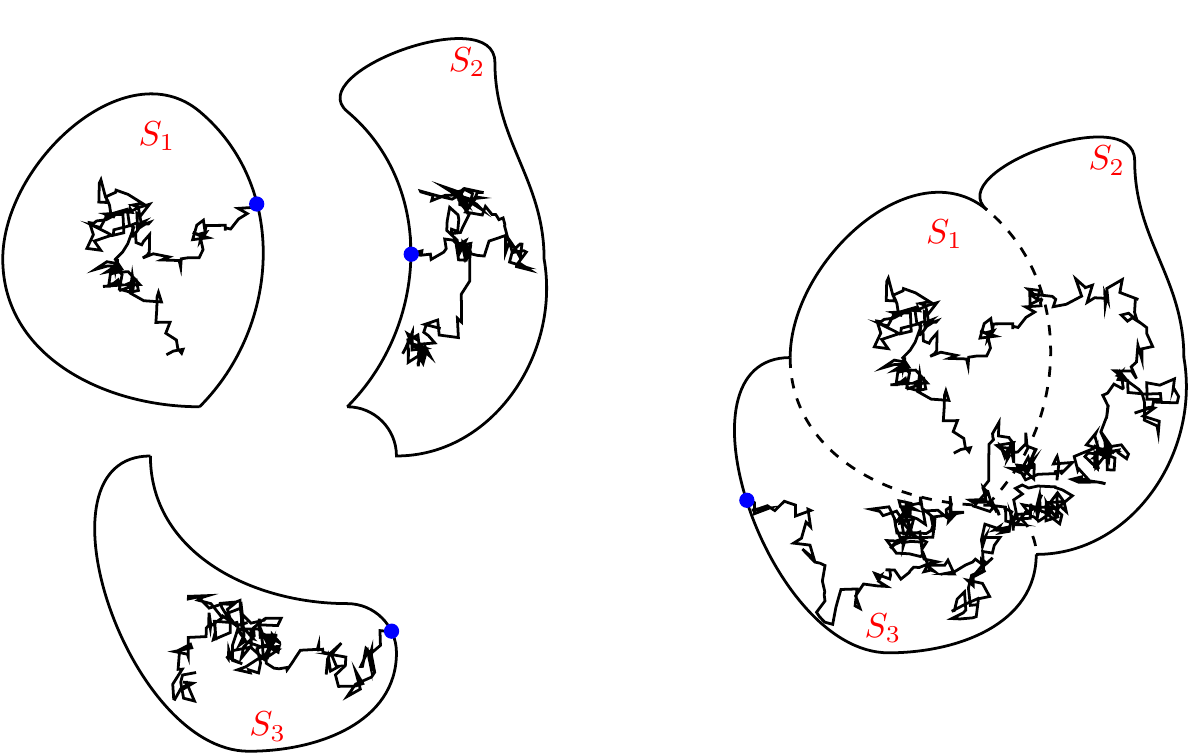}\caption{Processes before (on the left) and after  concatenation (on the right).}
\label{figurka}
\end{figure}
\end{center}

\vspace{-0.7cm}
Our general concatenation theorem, that is, Theorem \ref{thm5}, is devoted to the case where there is a finite number $N\ge 2$ of processes with values in 
disjoint compact topological spaces $S_i, i \in \{1,\dots,N\}$, and some of them are not honest. Those processes that are not honest have a finite, possibly greater than $1$, number of possible ways of exiting from their state-spaces (think, in particular, of a diffusion process on a finite interval which may exit from this interval by either of its end-points, see Section \ref{examples} for more examples). The extended process we construct in the state-space
\[ \suu \coloneqq \dot{\bigcup_{i=1,\dots, N}} S_i \]
(the dot above the union stresses the fact that $S_i$ are disjoint, `u' for `union') intuitively looks as follows (see Figure \ref{figurka}): when started at an $x \in S_i$ the process is initially identical to the original process defined in $S_i$; when the latter is no longer defined,
the extended process chooses a random point in $S$, the distribution of its new position depending both on the space it came from and the way the process has left it. Starting from this new position, the process forgets its past and, if its new position happens to be in $S_j$, behaves according to the rules governing the original process defined in $S_j$. Again, after a random time the original process can be undefined in $S_j$; then the extended process chooses a random point in $S$ and the choice once more depends on $j$ and on the way $S_j$ was exited, but not on the previous history, and so on.

There is a number of differences between our approach and that presented in \cite{werner}; for example, from the technical point of view, F. Werner works with general strong Markov processes whereas we restrict ourselves to Feller processes.  The most significant of differences, however, is that in \cite{werner} the order in which the processes are concatenated is deterministic: after behaving like $X^n$ the concatenated process starts anew as $X^{n+1}$. In our case, the order is stochastic: after exiting from $S_i$ the process chooses the point from which it starts anew randomly; moreover, its new position depends on the way it exited from $S_i$ --- this mechanism is described by Feller exit boundary for the process governing evolution in $S_i$ (see Section \ref{elae}).  

A very similar situation is studied in \cite{kostrykin2012}*{Section IV}, where two metric graphs are joined, and a Brownian motion on this larger graph is constructed by means of Brownian motions constructed beforehand on its components: a particle can filter from one subgraph to the other via so-called shadow vertices. It is clear that our approach fits such constructions better than that in \cite{werner}.


\subsection{The main idea}  

It should be stressed that, in the cases where there are many exits and many processes to concatenate, differential operators with domains described by transmission conditions of the type \eqref{intron:4}--\eqref{intron:5} are difficult to handle, and more often than not, one needs to develop special, frequently involved techniques to obtain appropriate generation theorems for the related stochastic processes. This was the case with generation theorems for diffusion processes on graphs \cites{bobmor,nagrafach,gregosiewicz,binz,gregosiewiczsem,gregnew,banfal2,marjeta} and in thin layers \cites{zmarkusem,zmarkusem2b,3Dlayers,sima}. 
The core of the problem was the fact that in these situations it is a priori unclear whether the so-called range condition is satisfied ---  the other two conditions characterizing generators (and pre-generators) of Feller semigroups are usually easy to check.

The main idea of the present paper is that in dealing with the problem of concatenation of processes that are not honest, the limitations of the previous approaches and technical problems which we face when applying them, disappear, if instead of working with generators we use \emph{resolvents} (see our Section \ref{pofs} for an explanation of these notions) and Laplace transforms of \emph{exit laws}. Namely, we show that in the rather general context described above, whether concatenation involves a change of boundary conditions or the way the involved generators act, the resolvent of concatenation of processes has the same, quite explicit form --- see \eqref{conc}--\eqref{concp}. Remarkably, all terms in the formula for the resolvent have a clear probabilistic interpretation, whereas the analytic methods used in the papers cited above seem to be unrelated to any stochastic intuitions.  It should be stressed here that this unification and simplification of the previous approaches is quite impossible  without the notion of exit law (see Section \ref{must}).  
 
Put otherwise: our formula for the resolvent of a concatenated process provides a simple way of checking the range condition for a candidate for a generator; in fact, it yields the form of the solution to the resolvent equation. Moreover, it is an interesting feature of formula \eqref{conc}--\eqref{concp} that it allows inferring properties of concatenation of processes from the properties of its constituents.

For example, the fact that the resolvent of the process that is obtained as a concatenation of other  processes has a quite explicit and manageable form, provides an efficient way to deal with convergence problems involving processes on graphs or similar structures, like those studied in \cite{bobmor,nagrafach,knigaz,banfal2,zmarkusem,gregosiewicz,gregosiewiczsem,gregnew}. 

To explain, we recall that the classical Trotter--Kato--Neveu theorem \cite{kniga,knigaz,ethier,goldstein,pazy} states that a sequence of semigroups converges to a limit semigroup iff their resolvents converge to the resolvent of the limit semigroup --- and this is equivalent to convergence in distribution of the corresponding processes (see \cite{kallenbergnew}*{p. 385}).  Since generation theorems obtained in \cites{bobmor,nagrafach,gregosiewicz,gregosiewiczsem,gregnew,banfal2,marjeta,binz} almost never provide explicit formulae for the resolvent, direct use the Trotter--Kato--Neveu theorem was so far impossible. Instead, quite involved manipulations on generators, based for example on the Sova--Kurtz version  of the approximation theorem \cite{sova,ethier,kniga,knigaz,kurtz} or asymptotic analysis (see e.g. \cite{banfal1}) had to be performed. As exemplified in our Section \ref{aap}, formula \eqref{conc} allows us to use the original version of the Trotter--Kato--Theorem, to prove a class of known convergence theorems with greater ease and to obtain new theorems of this type (see  further down for more details). 

\subsection{Structure of the paper} 
In Section \ref{pofs} we recall basic facts from the theory of Feller semigroups. Next, in Section \ref{owc}, we study the notion of exit laws. In the following Section \ref{elae}, a link between exit laws and excessive functions is recalled, and a definition of regular Feller boundary is given. The main Section \ref{cbnr} contains our master theorem (Theorem \ref{thm5}) saying that any finite number of Feller processes with regular Feller boundary can be concatenated by means of probability measures describing distributions of positions of points at which the process starts anew after exiting from one of the components of the state-space. Formula \eqref{concp} (or its more formal, but equivalent version \eqref{concb}) provides a quite explicit form of the resolvent of such a concatenated process. 

Section \ref{examples} contains examples of Feller processes with regular Feller bound\-ary, that is, of processes that can be concatenated within the theory presented in this paper. Notably, we provide exact formulae for their excessive functions and exit laws. These processes can be in particular used as building blocks in the theorem on fast processes on graphs we discuss in Section \ref{aap}. 

In Section \ref{asgt}, by presenting a sample generation theorem, we provide a link between the approach via resolvents developed in this paper, and the approach via generators, used before. 

The concluding Section \ref{aap} is devoted to a limit theorem, an averaging principle, describing in particular fast processes on graphs (see Section \ref{aasa} for more information); it is preceded by Section \ref{s:cos} where the theory of convergence of semigroups is briefly recalled.  In the Appendix (Section \ref{app}) supplementary material is collected.
 
\subsection{An averaging principle}\label{aasa}
As already mentioned, as a sample application of our general concatenation  theorem (Theorem \ref{thm5}), in Section \ref{aap},  we revisit the question of fast diffusions on graphs. To recall, in \cite{nagrafach}, following the previous work devoted to reconciling two models of evolution of neurotransmitters \cite{bobmor}, the following theorem was 
proved.

Consider a diffusion process on a graph $G$ without loops, and assume that each of its vertices is a semi-permeable membrane through which particles diffusing on the graph's edges can filter. Permeability of the membrane depends on the edge from which the particles filter and on the edge to which they filter; permeability differs also in two possible vertices by which the edges are connected. Now, assume that diffusion on each edge is accelerated and at the same time permeability of the membranes situated at vertices is lowered so that the fluxes through each of them remain constant.  In the limit, as diffusion is faster and faster, by averaging property of diffusion, points in each edge become indistinguishable and so each edge collapses to a single point, and the entire graph becomes a collection of disjoint points. Since the fluxes remain constant, however, these points communicate: the limit process is a Markov chain on the state-space composed of points formed from former edges. Intensities of jumps in the limit Markov chain are functions of fluxes of the approximating diffusions and thus functions of permeability coefficients.

There is a couple of other versions, extensions and generalizations of this result --- see \cite{banfal1,banfal2,zmarkusem,gregosiewicz,gregosiewiczsem,gregnew} where also a spectrum of applications to models of applied mathematics is discussed. In particular, Banasiak et al. \cite{banfal1} (see also \cite{bobmor}) expound the fact  that in biological systems often some interactions occur at a much faster time scale than other, and thus justify  the need for theorems of this type.

To explain the idea of the generalization we present in Section \ref{aap}, we consider the space of continuous functions on the closed interval $[0,r]$ and the operator $Af=\frac 12 f''$ with Feller--Wentzel boundary conditions (see \eqref{exa2:1}) --- processes generated by such operators, when concatenated, describe diffusions on graphs. If $c$ and $d$ in \eqref{exa2:1} are replaced by $\eps c$ and $\eps d$, respectively, and $A$ is replaced by $\eps^{-1}A$, in the limit as $\eps \to 0$ we obtain a process with the state-space composed of a single point (all points of the interval $[0,r]$ are lumped together). This limit process is described by two independent exponential random variables, say, $T_1$ and $T_2$. After $T\coloneqq \min (T_1,T_2)$  the process leaves its state-space and is no longer defined, if $T=T_1$ the process leaves through the trace of the left end of the interval, if $T=T_2$ it leaves through its right end --- see Section \ref{exa2} and Example \ref{ex:5} for details.

Such a limit property is shared by a large class of Feller processes, except that the number of exponential variables can be different that $2$ --- we single out such processes in Definition \ref{asik}. For instance, in the Walsh-type Skew Brownian motion on a star-like graph discussed in Section \ref{exa3} the number of variables coincides with the number of edges, and in a variety of other examples this number is $1$. Our averaging principle of Section \ref{aap} (Theorem \ref{thm:ap}) says that such processes, when concatenated, in the limit as $\eps \to 0$ converge to Markov chains.

\section{Preliminaries on Feller semigroups and resolvents}\label{pofs}  
\newcommand{\stare}{To recall (see e.g. \cites{bass,kniga,kallenbergnew}): if $S$ is a locally compact metrizable and separable  space, and $\coss$ is the space of continuous functions $f$ on $S$ that vanish at infinity (\ie for any $f\in \coss$ and $\eps >0$ there is a compact set $K$ such that $|f(x)|< \eps $ for $x \in S\setminus K$), then a Feller semigroup  
is a strongly continuous family of positive contraction operators $\rod{T(t)}$  in $\coss$ such that $T(0)=I_{C_0(S)}$ (the identity operator) and 
\[ T(t)T(s) = T(t+s) \qquad s,t \ge 0.\]
Because of the Riesz representation theorem, for each $x\in S $ and $t\ge 0$ there is a Borel measure $m_{x,t}$ on $S$ such that $m_{x,t} (S)\le 1$ and 
\begin{equation}T(t)f (x) = \int_S f(y) \ud m_{x,t}, \qquad f \in \coss \label{intro:1} \end{equation}
if $m_{x,t}(S) = 1$ for all $x$ and $t$, the process is said to be conservative.} 

To recall (see e.g. \cites{bass,kniga,kallenbergnew}): if $S$ is a compact, metrizable and separable  space, and $\coss$ is the space of continuous functions on $S$, then a Feller semigroup  
is a strongly continuous family of positive contraction operators $\rod{T(t)}$  in $\coss$ such that $T(0)=I_{C(S)}$ (the identity operator) and 
\[ T(t)T(s) = T(t+s) \qquad s,t \ge 0.\]
Because of the Riesz representation theorem, for each $x\in S $ and $t\ge 0$ there is a Borel measure $\textsf m_{x,t}$ on $S$ such that $\textsf m_{x,t} (S)\le 1$ and 
\begin{equation}T(t)f (x) = \int_S f \ud \textsf m_{x,t}, \qquad f \in \coss; \label{intro:1} \end{equation}
if $\textsf m_{x,t}(S) = 1$ for all $x$ and $t$, the process is said to be conservative or honest.

It is well-known that with each Feller semigroup one can associate a Markov  process \rod{X(t)} with c\`adl\`ag paths \cite{bass,kallenbergnew} (\ie paths that are right-continuous and possess left limits) in such a way that 
\begin{equation} \label{intro:2} T(t)f(x) = E_x f(X(t))1_{\{t\le \tau\}} , \qquad  f\in \coss, t\ge 0, x\in S\end{equation}
where $E_x$ denotes expectation conditional on $X(0)=x$ and $\tau$ is the lifetime of the process, that is, $\tau=\tau (\omega)$ is the random time up to which the path $X(t,\omega)$ is defined. (Comparing Eqs. \eqref{intro:1} and \eqref{intro:2} we see that $\textsf m_{x,t}(\Gamma), \Gamma \subset S$ should be interpreted as the probability that the related process starting at $x$ at time $0$ will be in $\Gamma$ at a time $t\ge 0$.)

Feller semigroups are conveniently described by their generators. The generator $A$ of a Feller semigroup is defined by 
\[ Af = \lim_{t\to 0} t^{-1} (T(t)f - f) \]
on the domain $\dom{A}$ composed of $f$ such that the above strong limit exists (that is, the right-hand side converges to $Af$ uniformly on $S$). It is well-known that the generator $A$ characterizes the semigroup uniquely (in particular, different semigroups have different generators) and that an operator $A$ is a generator of a Feller semigroup (shortly: a Feller generator) in $C(S)$ iff the following three conditions are met 
\begin{itemize}
\item [1. ] $A$ is densely defined, 
\item [2. ] $A$ satisfies the positive maximum principle,
\item [3. ] $A$ satisfies the range condition: for any $g \in C(S)$ and $\lam > 0$ there is an $f \in \dom{A}$ such that $\lam f - Af = g.$ \end{itemize} 
The semigroup generated by $A$ will in what follows be denoted $\sem{A}.$

In defining Feller semigroups it is customary to require also that the related process is honest or conservative (which comes down to the requirement that $1_S \in \dom{A}$ and $A1_S= 0$) but in this paper, for obvious reasons, we will allow the process to be dishonest. It is easy to see that the process is honest iff $T(t)1_S = 1_S$ for all $t\ge 0$ and this holds iff $\tau$ of equation \eqref{intro:1}, that is, the lifetime of the process, is a.s. equal to $\infty$.

An alternative description of a Feller semigroup is provided via Feller resolvents. A family $\rla, \lam >0 $ of non-negative operators in $\coss$ is said to be a Feller resolvent if the following conditions are met:  
\begin{itemize}
\item [1. ] the Hilbert equation holds:
\begin{equation}\label{owc:3} (\lam- \mu) \rmi \rla = \rmi - \rla, \qquad \lam,\mu >0, \end{equation}
\item [2. ] for each $f \in \coss$, $\lil \lam \rla f=f,$ 
\item [3. ] $\lam \rla 1_S \le 1_S$ for all $\lam >0$.
\end{itemize} 

It may be argued (see the already cited references) that for each Feller resolvent there is a Feller generator $A$ such that $\rla = \rez{A}$, and thus also the related Feller semigroup. Moreover, the related process is conservative iff $\lam \rla 1_S = 1_S$ for all $\lam >0$.

A word about terminology: a family $\rla , \lam >0$ of bounded operators is said to be a pseudoresolvent, if it satisfies the Hilbert equation. This family is said to be regular, if additionally condition 2. given above holds. 




\section{Exit laws} \label{owc}

Let $S$, as above, be a compact, metrizable and separable  space, and let $A$ be the generator of a Feller semigroup in $C(S)$ with resolvent $\rla, \lam >0$. 
By definition,  $\ell_\lam, \lam >0$ is the \emph{Laplace transform of an exit law} for $A$  iff (compare \cite{neveu62}*{p. 324}, \cite{fellerboundaryforMC}*{p. 538}, or \cite{fmc}*{Section 3.6.10}, cf. also the section on \emph{exit systems} in \cite{blum}; condition (c), below, bears remarkable resemblance to the characteristic equation in \emph{empathy theory} of N. Sauer -- see e.g. \cite{sauer}*{eq. (7)})  
\begin{itemize}
\item [(a) ] $(0,\infty)\ni \lam \mapsto \ell_\lam \in \coss$ is nontrivial (\ie $\ell_\lam \not = 0$ for at least one $\lam $), locally  bounded and non-negative, 
\item [(b) ]   $\lilz \ell_\lam (x)\le 1$ for each $x \in S$, and
\item [(c) ] we have 
\begin{equation}\label{owc:2} 
(\lam - \mu) \rla \ell_\mu = \ell_\mu - \ell_\lam , \qquad \lam,\mu >0.\end{equation}
\end{itemize}

In this section we collect basic properties of such objects. In what follows, they will, for simplicity, often be called `\emph{exit laws}'.

\begin{prop}\label{prop1} Each  locally bounded function $(0,\infty) \ni \lam \mapsto \ell_\lam \in \coss $ satisfying \eqref{owc:2} is infinitely differentiable with \begin{equation} \ell_\lam^{(n)} = (-1)^n n! (\rla)^n \ell_\lam , \qquad n \ge 1\label{owc:d} \end{equation}
 \end{prop}
\begin{proof} Fix a $\lam \in (0,\infty)$, and let $M$ be such that $\|\ell_\mu\|\le M$ for $\mu $ in a neighborhood  of $\lam$. Since $\|\lam \rla \|\le 1,$ the left-hand side of \eqref{owc:2} does not exceed $|\lam - \mu| \frac M\lambda $ in this neighborhood. It follows that, as $\mu \to \lam $, $\ell_\mu $ converges to $\ell_\lam$, establishing continuity of $\lam \mapsto \ell_\lam $.

To prove \eqref{owc:d} we proceed by  induction. For $n=1$ the formula is a direct consequence of \eqref{owc:2} and the already established continuity. Moreover, assuming \eqref{owc:d} is true for some $n$, we have 
\begin{align*}
\frac {\ell_\mu^{(n)} - \ell_\lam ^{(n)}}{\mu - \lam } &= \frac {(-1)^n n!}{\mu - \lam } \left [ (\rmi)^n \ell_\mu - (\rla)^n \ell_\lam \right ]\\ &= (-1)^n n! \left [ (\rmi)^n \left (\frac {\ell_\mu - \ell_\lam}{\mu - \lam}\right ) + \left (\frac {(\rmi)^n - (\rla)^n}{\mu -\lam} \right ) \ell_\lam  \right ].  
\end{align*} Since $(\rla)' = - (\rla)^2$, using the already checked case $n=1$, we obtain 
\begin{align*} \ell_\lam^{(n+1)} & = (-1)^n n! (\rla )^n \ell_\lam' + (-1)^{n+1} n! n (\rla)^{n-1}(\rla)^2 \ell_\lam \\ &= (-1)^{n+1}  (n+1)!(\rla)^{n+1} \ell_\lam, \end{align*} as desired.  
\end{proof} 

To recall, a real-valued, non-negative function  defined on $(0,\infty)$ is said to be absolutely monotone iff it is infinitely differentiable and for each $n$, its $(n+1)$st derivative has an opposite sign to its $n$th derivative. The Bernstein Theorem says that  a function is absolutely monotone iff it is the Laplace transform of a Borel, possibly infinite,  measure on $[0,\infty)$ (see \cite{feller}*{p. 439}). Thus, Proposition \ref{prop1} implies that 
if $\ell_\lam $ is the Laplace transform of an exit law, then
for each $x\in S$, $\lam \mapsto \ell_\lam (x)$, being absolutely monotone and satisfying $\lilz \ell_\lam (x)\le 1$,  is the Laplace transform of a Borel sub-probability measure,  say $m_x$, on $[0,\infty)$:
\begin{equation}\label{owc:dod} \ell_\lam (x)= \int_0^\infty \e^{-\lam t} \, m_x (\mud t), \quad \lam >0. \end{equation} 
We note that, for a Borel subset $\mathcal B$ of $[0,\infty)$, $m_x(\mathcal B)$ should be interpreted as the probability that the process generated by $A$ and starting at $x$ will leave $S_A$ at time $t\in \mathcal B$.

\begin{prop}\label{prop2} None of the measures $m_x, x\in S_A$ has an atom at $t=0$. \end{prop} 
\begin{proof} The resolvent $\rla, \lam >0$ is regular, that is, $\lil \lam \rla f = f, f \in \coss$. This implies that the left-hand side of \eqref{owc:2} converges to $\ell_\mu$, as $\lam \to \infty$.  On the other hand, by \eqref{owc:dod}, $\lil \ell_\lam (x) = m_x(\{0\})$, showing that the value of the right-hand side of \eqref{owc:2} at a point $x\in S$ converges to $\ell_\mu (x) - m_x(\{0\}).$ These two limits cannot be reconciled unless $m_x(\{0\})=0.$ 
 \end{proof}
\begin{cor}\label{cor1} As a by-product of the proof of Proposition \ref{prop2}, we obtain $\lil \ell_\lam = 0$ (in the norm of $\coss$). \end{cor}

Our next theorem establishes existence of an exit law for any non-conservative (not honest) Feller generator $A$.
As already recalled, a Feller generator $A$ with resolvent $\rla , \lam >0$ is conservative iff $\lam \rla 1_{S} = 1_{S}$ for each $\lam >0.$ 
Since $\|\lam \rla\|\le 1$ and $\rla \ge 0$, we see that for nonconservative $A$  the function $1_{S} - \lam \rla 1_{S} \in \coss$ is non-negative and nontrivial (\ie it does not vanish everywhere), for all $\lam >0$.

\begin{thm}\label{thm4} Suppose $A$ is not conservative. Then, there exists at least one exit law for $A$. 
 \end{thm} 
\begin{proof} 
By assumption 
\[ \ell_\lam \coloneqq  1_{S} - \lam \rla 1_{S}\ge 0, \qquad \lam >0 \]
is non-zero and $\|\ell_\lam \|\le 1.$ Moreover,
\[ \ell_\mu - \ell_\lam = \lam \rla 1_{S} - \mu \rmi 1_{S},\]
and, because of the resolvent equation for $\rla, \lam >0$,  
\begin{align*}
 (\lam - \mu) \rla \ell_\mu & = (\lam - \mu) \rla 1_{S} + \mu (\mu - \lam )\rla \rmi 1_S \\ & =  (\lam - \mu) \rla 1_{S} +\mu (\rla - \rmi )1_{S} \\ &= \lam \rla 1_{S} - \mu \rmi  1_{S}. \end{align*}
This establishes \eqref{owc:2}, and in particular implies representation \eqref{owc:dod}. This representation in turn, when combined with $\|\ell_\lam\|\le 1$, implies that the limit $\lilz \ell_\lam (x)$ exists and does not exceed $1$ for each $x\in S$.  
\end{proof}
In what follows, we will write $\dell$ to denote the particular Laplace transform of the exit law found in Theorem \ref{thm4}: 
\[ \dell \coloneqq 1_{S} - \lam \rla 1_{S}.\]

\begin{ex}\label{ex0} Let $S=[0,1]$ be the unit interval, and let $Af=f'$ on the domain composed of continuously differentiable $f$ such that $f'(1)+ \alpha f(1)=0,$ where $\alpha\ge 0$ is a given parameter. In the process related to $A$, a particle starting at an $x\in [0,1)$ moves to the right with constant velocity $v=1$; upon reaching $x=1$ it stays there for an exponential time with parameter $\alpha$, and after this exponential time elapses, the process is no longer defined (for $\alpha =0$ the process stays at $x=1$ for ever, and in particular is well-defined for all $t\ge 0$).  For this operator, the semigroup may be given explicitly: 
\[ \e^{tA} f(x) = \begin{cases} f(x+t), & t\le 1-x, \\
\e^{-\alpha (t-1 +x)} f(1),&  t > 1- x.\end{cases}\]
It follows that 
\begin{equation}\label{exa0:1} \rla f(x) = \int_x^1 \e^{\lam (x-y)} f(y) \ud y +\e^{\lam (x-1)} \frac {f(1)}{\lam + \alpha}. \end{equation}
Thus, here, (if $\alpha >0$) 
\[ \dell (x) = 1 - \lam \rla 1_{S}(x) = \e^{\lam (x-1)} \frac {\alpha}{\lam + \alpha}, \qquad x\in [0,1]. \]

The first factor in this product is the Laplace transform of the deterministic time $T=1-x$ needed for the process starting at an $x\in [0,1]$ to reach $x=1$; the other factor is the Laplace transform of the exponential time spent at the boundary $x=1$. Thus, all in all, $\della (x)$ is the Laplace transform of the lifetime of the process governed by $A$ that starts at $x\in [0,1].$ \qed 
\end{ex} 

The example presented above is a particular case of the more general result. Namely, we have the following observation. 
\begin{prop} Let $\tau$ be the lifetime of the process governed by a Feller resolvent $\rla, \lam>0$. Then $\dell $ is the Laplace transform of $\tau $:
\[ \dell (x) = E_x \e^{-\lam \tau}, \qquad x \in S.\]
 \end{prop}
\begin{proof} By \eqref{intro:2}, $\e^{tA}1_{S} (x) = \pr_x (t\le \tau)$. On the other hand,
\begin{align*} E_x\e^{-\lam \tau}& = \lam \int_0^\infty \e^{-\lam t} \pr_x (\tau \le t) \ud t= 1 - \lam \int_0^\infty \e^{-\lam t} \pr_x (t \le \tau) \ud t \\
&= 1 - \lam \int_0^\infty \e^{-\lam t} \e^{tA} 1_{S} (x)  \ud t = 1 - \lam \rla 1_S(x),\end{align*}
as desired. \end{proof}

\begin{cor}\label{niczonko} We have 
\[ \max_{x\in S} \dell (x) < 1.\] 
\end{cor}
\begin{proof} For no point $x\in S$ can we have $\dell (x) =1$ for this would mean that the lifetime of the process starting at $x$ is $0$, contradicting Proposition \ref{prop2}. Since $S$ is compact and $\dell $ is continuous, we are done. \end{proof}

\begin{ex}\label{ex01} Let $A$ be an honest Feller generator. Then, for any $\alpha > 0,$ $A_\alpha \coloneqq  A - \alpha I$ where $I$ is the identity operator, is a not honest Feller generator. The process related to $A_\alpha$ is identical to that related to $A$ up to an independent exponentially distributed time $T$ with parameter $\alpha $; from this point on, the process generated by $A_\alpha$ is undefined. In other words, $T$ is the lifetime of the modified process. Indeed, since $A$ is honest,
\[ 1_S - \lam \rez{A_\alpha} 1_S =  1_S - \lam \left (\lam + \alpha - A\right )^{-1} 1_S = {\textstyle \frac \lam{\lam + \alpha}} 1_S,\]
and $\lam \mapsto  \frac \lam{\lam + \alpha}$ is the Laplace transform of the distribution of $T$. 
\end{ex}

\newcommand{\nad}{\phi}

\section{Exit laws and excessive functions; Feller boundary}\label{elae} 
Let again $A$ be a Feller generator in $C(S)$ with resolvent $\rla , \lam >0.$ A  nonnegative function $\nad \in C(S) $ such that $\nad \le 1_{S}$ and 
\begin{equation}\label{exces} \lam \rla \nad \le \nad, \qquad \lam >0 \end{equation} 
is  termed \emph{excessive}, see \cites{bass,beznea,chunglectures,doobc,revuz}. We will say that an excessive $\nad$ is nontrivial if $\lam \rla \nad \not = \nad$ for at least one $\lam $ (and thus for all $\lam >0$). 
\begin{rem} In most monographs, the requirement that $\nad \in \coss$ is omitted in the definition of excessive function; instead other, milder, regularity conditions are imposed. In this paper, the assumption that the Laplace transform of an exit law is a member of $\coss$ seemed more pleasing.\end{rem}

As it turns out, many Laplace transforms of exit laws may be constructed from nontrivial excessive functions. Namely, calculations presented in Theorem \ref{thm4} show that \begin{equation} \label{edefef} \ell_\lam \coloneqq  \nad - \lam \rla \nad \end{equation} satisfies \eqref{owc:2}, and thus (being bounded by $1_{S}$) is the Laplace transform of an exit law as long as $\nad$ is excessive and nontrivial.

Conversely, suppose we are given the Laplace transform of an exit law $\ell_\lam, \lam >0$. Since, for each $x\in S$, $\ell_\mu (x)$ increases as $\mu \searrow 0$, and $\rla $s are integral operators, we may consider \begin{equation}\label{defef} \nad (x) \coloneqq \lim_{\mu \to 0} \ell_\mu (x)\le 1 .\end{equation} Then, letting $\mu \to 0$ in \eqref{owc:2} we obtain (by the Monotone Convergence Theorem), 
\begin{equation}\label{repra} \ell_\lam (x) = \nad (x) - \lam \rla \nad (x), \qquad x\in S, \lam >0.\end{equation}
It should be stressed, though, that in this representation of $\ell_\lam, \lam >0$, in contrast to \eqref{edefef}, $\nad $ need not be continuous. Nevertheless, the right-hand side  makes sense because the operator $\rla$, being related to a measurable kernel, has a natural extension to the space of bounded measurable functions. 

Even if we restrict ourselves to  $\nad \in \coss$, there still remains the question of uniqueness of representation of exit laws in terms of excessive functions. Fortunately, there is a straightforward test for such uniqueness: to make sure that $\nad$ in \eqref{repra} is determined by $\ell_\lam, \lam >0$ it suffices to check that there are no non-zero $f\in \coss $ such that $f =\lam \rla f $, that is, no non-zero $f\in \dom{A}$ such that $Af=0$.

In what follows we restrict ourselves to exit laws of the form \eqref{edefef} with $\nad \in \coss$ and assume that the kernel of $A$ is trivial. Moreover, since, 
as we shall see in the next section, our work is motivated by Feller processes that have more than one exit law, we will in fact be concerned with the situation described by  the following definition (comp. \cite{fellerbi} or \cite{fmc}*{Section 3.7}). 

\begin{defn}\label{defn:fb} \emph{We will say that the process governed by $A$ has a regular Feller exit boundary with $\ka$ exit points, if 
\begin{itemize}
\item [(i) ] the kernel of $A$ is trivial (that is, composed of zero function merely),
\item [(ii) ] there are nontrivial excessive functions $\nad^j, j\in \{1,\dots, \ka\}$ for $A$ such that \begin{equation}\label{repre} \sum_{j=1}^\ka \nad^j = 1_{S}.\end{equation}
\end{itemize}} \end{defn}
The Laplace transform of the exit law $\ell_\lam ^j$ related to $\nad^j$: 
\[ \ell_\lam^j \coloneqq \nad^j - \lam \rla \nad^j \]
will be somewhat informally called the law for exit through the $j$th exit (or: $j$th gate).

We stress again that $\nad^j$ are, by definition, continuous and that, by assumption (i), representation in terms of excessive functions is unique. It should be noted, however, that our definition says nothing about uniqueness of representation \eqref{repre}: it is possible, in particular, that one  of $\nad^j$s is a sum of two smaller nontrivial excessive functions (compare discussion in \cite{fellerbi} or \cite{fmc}*{Section 3.7}). Therefore, the definition should be understood to mean that the process has \emph{at least} $\ka$ ways of exiting its state-space, and it is in this sense that this definition is applied in what follows. Nevertheless, in the examples of Section \ref{examples} it will be clear from the probabilistic description of the processes involved that there are precisely one, two or $\ka$ exits there. 

\section{Concatenation of  $N$ processes}\label{cbnr} 

\subsection{Set-up}\label{setup}
Suppose we are given $N$ Feller processes in $N$ separate (i.e., disjoint) compact spaces $S_1,\dots, S_N$, generated by operators $A_1,\dots, A_N$, respectively. To avoid trivialities, we assume that at least one of these generators is not conservative. Since $S_i$s can be reordered if necessary, we assume without loss of generality that
the first $M\in\{1,\dots,N\}$ processes are not conservative, whereas the remaining ones are conservative. 
Moreover, we assume that the process governed by $A_i, i =1,\dots ,M$ has a regular Feller boundary with $\ka (i)$ exit points: there are excessive functions $\nad^{i,j}, j= 1,\dots, \ka (i)$ such that 
\begin{equation}\label{brzegi} \sum_{j=1}^{\ka (i)} \nad^{i,j} = 1_{S_i}, \qquad i =1,\dots , M. \end{equation} 
For the corresponding exit laws
\[ \ell_\lam^{i,j} \coloneqq \nad^{i,j} - \lam \rlai \nad^{i,j} \]
this means that 
\begin{equation}\label{brzegip} \sum_{j=1}^{\ka (i)} \ell_{\lam}^{i,j} = 1_{S_i} - \lam \rlai 1_{S_i} \eqqcolon \dell^i,  \qquad i =1,\dots , M. \end{equation}

We construct the resolvent of a Feller process on 
\[ \suu \coloneqq \dot{\bigcup_{i=1,\dots,N}} S_i ,\]  
(the dot stresses the fact that $S_i$ are disjoint; `u' for `union'), which is governed by the  following rules. 
\begin{itemize}
\item [(i)] Conditional on starting at an $x\in S_i$, the process is identical to that related to $A_i$ up to the random time when the latter is no longer defined (if $A_i$ is conservative, this random time is $\infty$ and thus both processes are identical at all times). 
 \item [(ii)] If and when the process related to $A_i$ is no longer defined, the process we construct starts afresh at a random point $y \in \suu$, the distribution of its position at this moment depending on the gate through which $S_i$ was exited. More specifically, this position is described by a Borel sub-probability measure $\mem_{i,j}$ on $\suu$. In most applications, the latter measure is supported in $\suu \setminus S_i$, but this assumption is not needed in this section.  
 \item [(iii)] If $y$ belongs to $S_j$, the constructed process is identical to that governed by $A_j$ up to the random time, when the latter is no longer defined, and so on. 
\end{itemize} 

\subsection{More notation}
To proceed, we need to establish more notation. First of all, to shorten formulae, we introduce
\begin{align*} \mc N & \coloneqq \{1,\dots, N\}, \\ 
\mc M & \coloneqq \{1,\dots, M\},\\
\mc {IJ} & \coloneqq \{ (i,j); i \in \mc M, j \in \{1,\dots ,\ka(i)\} \}.   
\end{align*}
Secondly, it will be convenient to think of an $f\in \cosi$ as a member of $\cosu$ also: to this end we extend $f$ to the entire $\suu$ by agreeing that it equals zero outside of $S_i$. This allows thinking of sums of  $f_i \in C(S_i)$ as members of $\cosu$: 
\[ f \coloneqq \dot{\sum_{i\in \mc N}} f_i ;\]
as above, the dot stresses the fact that we add together functions that are defined on separate spaces. Conversely, for an $f \in \cosu$ we will denote its part in $\cosi$ by $f_i$, that is, $f_i\coloneqq f 1_{S_i}$ or, with the identification described above in mind, $f_i = f_{|S_i}.$

Let $\rlai, \lam >0$ denote the resolvent of the $i$th process: \[ \rlai = \rez{A_i}, \qquad \lam >0, i\in \mc N. \] 
The formula
\begin{equation}\label{du}  \rladu f = \isuma \rlai f_i , \qquad \lam >0, f \in \cosu \end{equation}
describes the resolvent of the process called \emph{disjoint union} of the involved process (see \cite{sharpe}*{p. 83}). In this union there are no jumps between spaces $S_i,i \in \mc N$: if random evolution is started at $x\in S_i$, the rules of $A_i$ are obeyed throughout the process's lifetime. When a particle exists $S_i$, the process is left undefined; it never jumps to any other $S_j$ or back to $S_i$. We note the following connection between the exit laws of the right-hand side of  
\eqref{brzegip} and $\rladu, \lam >0$: 
\begin{equation} \isumap{\dell^i} = 1_{\suu} - \lam \rladu 1_{\suu} \eqqcolon \dell. \label{sumdel} \end{equation}

\subsection{The definition: first attempt}\label{td:fa} 

By the \emph{concatenation of processes} governed by $A_i, i \in \mc N$ we understand the process with the following resolvent: 
\begin{equation}\rlac f = \isuma [ \rla^i f_i + \sum_{j=1}^{\ka(i)} \left ( \int_{S} \rlac f \ud \mem_{i,j} \right )\ell_\lam^{i,j} ], \qquad \lam >0, f \in \cosu; \label{conc} \end{equation}
for $i \in \mc N\setminus \mc M$, the inner sum above is understood to be zero. The Borel measure $\mem_{i,j}$  which, by assumption, satisfies $0\le \mem_{i,j} (\suu) \le 1$, is the distribution of the point from which the process starts anew after exiting from the $i$th region through the $j$th gate. For some of $\mem_{i,j}$ we may have $\mem_{i,j} (\suu) < 1$,  some of $\mem_{i,j}$ might in fact be zero. Of course, if all of them are zero, $\rlac, \lam >0$ coincides with the resolvent $\rladu, \lam >0$ of the disjoint union process.

We note also that, because of the convention concerning the inner sum, formula \eqref{conc} may equivalently be written as 
\begin{equation}\label{concp} \rlac f =  \rladu f  + \isumap \sum_{j=1}^{\ka(i)} \left ( \int_{S} \rlac f \ud \mem_{i,j} \right )\ell_\lam^{i,j} , \end{equation}
where $\rladu , \lam >0$ is the resolvent of the disjoint union, defined in \eqref{du}.

\subsection{Consistency condition}
The first and obvious problem with \eqref{conc} or its equivalent form \eqref{concp} is that these formulae define $\rlac $ via yet undefined integrals 
\begin{equation}\label{calki} \int_{S} \rlac f \ud \mem_{i,j}, \qquad (i,j) \in \mc{IJ}.\end{equation}
Fortunately, these may be calculated beforehand from consistency condition for \eqref{conc}, which we will now derive.

To this end, we change summation indexes $i$ and $j$ in \eqref{concp} to $k$ and $l$, respectively, then fix $i\in \mc M$ and $j\in \{1,\dots, \ka (i)\}$, and integrate both sides of equation \eqref{concp} with respect to $\mem_{i,j}.$ Since $\ell_\lam ^{k,l}$ is supported in $S_k$, we obtain
\begin{equation}\label{cons} 
\int_{\suu} \rlac f \ud \mem_{i,j} = \int_{\suu} \rladu f \ud \mem_{i,j} + {\sum_{k\in\mc M}} \sum_{l=1}^{\ka (k)}  \int_{\suu} \rlac f \ud 
\mu_{k,l} \int_{S_k} \ell_\lam^{k,l} \ud \mem_{i,j}.\end{equation}
We claim that this consistency condition determines integrals \eqref{calki} as functions of 
\[ \int_{\suu} \rladu f \ud \mem_{i,j}, \quad (i,j) \in \mc{IJ} \mquad{ and } \int_{S_k} \ell_\lam^{k,l} \ud \mem_{i,j}, \quad (i,j),(k,l)\in \mc{IJ},\]
and these may be treated as known prior to defining $\rlac f$.

In the proof of our claim, we consider any $u \in \R^\ka $ where \[ \ka = \ka (1) + \dots + \ka (M),\] as the concatenation (in the linear algebra sense) of vectors 
\[
\begin{matrix} &(u_{1,1}& u_{1,2}&  \dots & u_{1,\ka(1)}) \in \R^{\ka (1)},\\
&  \vdots  & \vdots & \dots & \vdots  \phantom{\in \R^{\ka (1)},} \\ &(u_{M,1}& u_{M,2}&  \dots & u_{M,\ka(M)}) \in \R^{\ka (M)},
\end{matrix}\]
and equip $\R^\ka$ with the maximum norm $\|u\|= \max_{(i,j)\in \mc{IJ}} |u_{i,j}|$.  
 
\begin{lem}\label{cnp:lem1} Fix $\lam >0$. For any $v = \left ( v_{i,j} \right )_{(i,j) \in \mc{IJ}} \in \R^\ka $ there is precisely one $u  =  \left ( u_{i,j} \right )_{(i,j) \in \mc{IJ}}\in \R^\ka $ such that
\begin{equation}\label{mapka} u_{i,j} = v_{i,j} +  {\sum_{k\in\mc M}} \sum_{l=1}^{\ka (k)}u_{k,l} \int_{S_k} \ell_\lam^{k,l} \ud \mem_{i,j}, \qquad (i,j) \in \mc{IJ}.\end{equation}
Furthermore, $u\ge 0$ whenever $v\ge 0$. 
 \end{lem}
\begin{proof} It suffices to show that the linear map, say, $N_\lam$, that assigns
\[\left (   {\sum_{k\in\mc M}} \sum_{l=1}^{\ka (k)}w_{k,l} \int_{S_k} \ell_\lam^{k,l} \ud \mem_{i,j} \right  )_{(i,j)\in \mc {IJ}} \in \R^\ka \]
to a $ w \coloneqq (w_{i,j})_{(i,j)\in \mc {IJ}}$, has norm smaller than $1$. For, then the series $\sum_{n=0}^\infty N_\lam^n$ converges and $u\coloneqq (\sum_{n=0}^\infty N_\lam^n )v$ is the unique solution to \eqref{mapka}. Moreover, $u\ge 0$ provided that $v\ge 0$ because $N_\lam$ is a non-negative operator. 

On the other hand, obviously, 
\[ \|N_\lam w \| \le \|w\| \max_{(i,j)\in \mc{IJ}} \sum_{k\in \mc M} \sum_{l=1}^{\ka (k)} \int_{S_k} \ell_\lam^{k,l} \ud \mem_{i,j} ,\]
whereas, by \eqref{brzegip} in the first step and \eqref{sumdel} in the last, 
\begin{equation}\label{niezlasumka}  \sum_{k\in \mc M} \sum_{l=1}^{\ka (k)} \int_{S_k} \ell_\lam^{k,l} \ud \mem_{i,j} =   \sum_{k\in \mc M} \int_{S_k} \dell^k \ud \mem_{i,j} = \int_{\suu} \dot{\sum_{k\in \mc M}} \dell^k \ud \mem_{i,j} = \int_{\suu} \dell\ud \mem_{i,j}.\end{equation}
This means that 
\[ \|N_\lam\| \le \max_{(i,j)\in \mc{IJ}}  \int_{\suu} \dell\ud \mem_{i,j}.\]
Now, since $\mem_{i,j}$s are sub-probability measures, Corollary \ref{niczonko} tells us that $0\le \int_{\suu} \dell \ud \mem_{i,j} < 1$ for each $(i,j)\in \mc{IJ}.$ This completes the proof, the set $ \mc{IJ}$ being finite. 
\end{proof}
In what follows we write $M_\lam \coloneqq  I - N_\lam$ where $I$ is the identity operator in $\R^\ka$. Our lemma says that the  operator $M_\lam$ is invertible with
\[ M_\lam^{-1} \coloneqq \sum_{n=0}^\infty N_\lam^n .\]

\subsection{Back to the definition} \label{bttd}
Lemma \ref{cnp:lem1} allows filling the hole in the definition given in Section \ref{td:fa} as follows. 
For $\lam >0$ and $f\in \cosu$ we define vectors $v=v(f,\lam)$ and $u=u(f,\lam)$ in $\R^\ka$ by
\[ v = \left (\int_{S} \rladu g \ud \mem_{i,j} \right )_{(i,j) \in \mc {IJ}} \mquad{
and } u = (u_{i,j}(f,\lam))_{(i,j) \in \mc {IJ}}  = M_\lam^{-1} v, \] 
and then let $\rlac f$ be defined by (cf. \eqref{concp})
\begin{equation}\label{concb} \rlac f =   \rladu f  + \isumap \sum_{j=1}^{\ka(i)} u_{i,j}(f,\lam) \ell_\lam^{i,j}.
 \end{equation}
Since $\ell^{i,j}_\lam$s are continuous functions, $\rlac $ maps $\cosu$ to $\cosu,$ and it is clear that it is a linear operator.

In the next section, we will show that the so-defined operators $\rlac, \lam >0$ form a Feller resolvent. Here we would like to argue merely that, for the so-defined $\rlac f$, the vector of integrals $\int_{\suu} \rlac f \ud \mem_{i,j}, (i,j) \in \mc {IJ}$ coincides with $u(f,\lam):$
\begin{equation}\label{bttd:1} \widetilde u \coloneqq \left ( \int_{\suu} \rlac f \ud \mem_{i,j} \right )_{(i,j) \in \mc {IJ}} =  \left ( u_{i,j}(f,\lam) \right )_{(i,j) \in \mc {IJ}} \end{equation}
 Indeed, integrating \eqref{concb} over $\suu$ with respect to $\mem_{i,j}$, we check that 
\[ \widetilde u = v + N_\lam u . \]
But $u$ is the only solution to the equation $u=v +N_\lam u$, where $v$ is given, and so, in particular, $v + N_\lam u$ equals $u$. This establishes the claim.

\subsection{$\rlac, \lam >0$ is a Feller resolvent in $\cosu$}

For the proof that $\rlac, \lam >0$ is a Feller resolvent, we need two lemmas. The first of these will help us prove that $\rlac , \lam >0$ satisfies the Hilbert equation, the second will be instrumental in proving $\lam \rlac 1_{\suu} \le 1_{\suu}.$   
\begin{lem}\label{cbnrl2} For $\lam ,\mu>0$ and $f \in \cosu$,
\[ (\lam - \mu) u(\rlac f,\mu) = u(f,\mu) - u(f,\lam). \]
In other words, 
\[ \int_{\suu} (\lam - \mu)\rmic \rlac f \ud \mem_{i,j} = \int_{\suu} \rmic f \ud \mem_{i,j} -  \int_{\suu} \rlac f \ud \mem_{i,j} , \qquad (i,j) \in \mc {IJ}. \]
\end{lem}
\begin{proof}
We have 
\begin{align*}
u(\rlac f ,\mu) &=M_\mu^{-1} \left (\int_{S} \rmidu (\rlac f) \ud \mem_{i,j} \right )_{(i,j) \in \mc {IJ}} \\
&=M_\mu^{-1} \left (\int_{S} [\rmidu  \rladu f  + \rladu \isumap \sum_{j=1}^{\ka(i)} u_{i,j}(f,\lam) \ell_\lam^{i,j}] \ud \mem_{i,j} \right )_{(i,j) \in \mc {IJ}} \\
&=M_\mu^{-1} \left (\int_{S} [\rmidu  \rladu f  + \isumap \sum_{j=1}^{\ka(i)} u_{i,j}(f,\lam) \rlai \ell_\lam^{i,j}] \ud \mem_{i,j} \right )_{(i,j) \in \mc {IJ}},
\end{align*} 
with the last equality by \eqref{du}. 
Therefore, the resolvent equation and the definition of the Laplace transform of an exit law show that $(\lam - \mu)u(\rlac f,\mu)$ equals
\begin{align*}
&M_\mu^{-1} \left ( \int_{S}\rmidu f  \ud \mem_{i,j}  \right )_{(i,j) \in \mc {IJ}}
\\&-  M_\mu^{-1} \left (\int_{S} [ \rladu f   +  \isumap \sum_{j=1}^{\ka(i)} u_{i,j}(f,\lam) ( \ell_\lam^{i,j} - \ell_\mu^{i,j})]  \ud \mem_{i,j} \right )_{(i,j) \in \mc {IJ}} \\
=& \, u(f,\mu) - M_\mu^{-1} \left (\int_{S}  \rlac f \ud \mem_{i,j}   -  \sum_{i\in \mc M} \sum_{j=1}^{\ka(i)} u_{i,j}(f,\lam)\int_{S_i} \ell_\mu^{i,j}   \ud \mem_{i,j} \right )_{(i,j) \in \mc {IJ}}.\end{align*}
Since, as established at the end of Section \ref{bttd}, $\int_{S}  \rlac f \ud \mem_{i,j}$ is the same as $u_{i,j} (f,\lam)$,
the second term here equals
$ M_\mu^{-1} M_\mu u (f,\lam) = u(f,\lam)$. This completes the proof. 
 \end{proof}

\begin{lem}\label{pdojs} For $\lam >0$, 
\[ \lam u(1_{\suu}, \lam) \le \left ( \mem_{i,j} (\suu) \right )_{(i,j)\in \mc {IJ}}. \]
Moreover, the inequality can be replaced by equality provided that $\mem_{i,j} (\suu) =1$ for all $(i,j) \in \mc {IJ}.$ 

 \end{lem}
\begin{proof} Let 
\[ v = v(1_{\suu},\lam) = \left ( \int_{\suu} \rladu 1_{\suu} \ud \mem_{i,j} \right )_{(i,j)\in \mc {IJ}},\]
so that, by definition $u(1_{\suu}, \lam) = (\sum_{n=0}^\infty N_\lam^n ) v.$ We have, in other words, $u(1_{\suu},\lam) = \gra u^n $, where $u^0 = v$ and $u^{n+1} = v + N_\lam u^n, n \ge 1.$ It suffices, therefore, to show that
\begin{equation}\label{indukcyjne} \lam u^n \le \left ( \mem_{i,j} (\suu) \right )_{(i,j)\in \mc {IJ}}, \end{equation}
and we will do this by induction argument.

This condition holds for $n=0$, because $\lam \rladu 1_{\suu} \le 1_{\suu}.$ Moreover, writing 
\[ u^n = \left ( u_{i,j}^n \right )_{(i,j)\in \mc {IJ}}, \qquad n \ge 0, \] 
and assuming that \eqref{indukcyjne} holds for an $n$, we obtain 
\begin{align} \lam u^{n+1}_{i,j} & = \int_{\suu} \lam \rladu 1_{\suu} \ud \mem_{i,j} + \sum_{k\in \mc M} \sum_{l=1}^{\ka (k)} u_{i,j}^n \int_{S_k} \ell_\lam^{k,l} \ud \mem_{i,j} \nonumber \\
& \le  \int_{\suu} \lam \rladu 1_{\suu} \ud \mem_{i,j} + \sum_{k\in \mc M} \sum_{l=1}^{\ka (k)}  \int_{S_k} \ell_\lam^{k,l} \ud \mem_{i,j}, 
\label{niezlasumkap} \end{align}
because, clearly, $u^n_{i,j} \le \mem_{i,j}(\suu)\le 1.$ As proved in \eqref{niezlasumka}, the double sum here equals $\int_{\suu} \dell \ud \mem_{i,j}$. Hence, 
\[ \lam u^{n+1}_{i,j} \le \int_{\suu} [\lam \rladu 1_{\suu} +\dell ]\ud \mem_{i,j}=  \int_{\suu} 1_{\suu} \ud \mem_{i,j} = \mem_{i,j}  (\suu),\]
as desired. 

We are thus left with proving the last assertion. To this end, we note that, as long as all $\mem_{i,j} (\suu)$ are equal to $1$, the vector $u\coloneqq  \left ( \lam^{-1} \right )_{(i,j)\in \mc {IJ}}$ satisfies 
\begin{equation}\label{uv} u = v + N_\lam u, \end{equation}
for $v$ defined at the beginning of our proof. Indeed, the $(i,j)$th coordinate of the right hand side is equal to the expression in the second line of \eqref{niezlasumkap}, divided by $\lam .$  As proved below \eqref{niezlasumkap}, this expression further reduces to $\frac 1\lam \mem_{i,j}  (\suu)=\frac 1\lam $. But $u(1_{\suu},\lam)$ is the unique  solution to equation \eqref{uv}, and so we must have $u(1_{\suu},\lam) = u$.  
This shows the claim.
\end{proof}

We are finally ready to prove our main theorem in this section, and a fundamental result for all that follows.  

\begin{thm}\label{thm5} The family $\rlac , \lam >0$ defined by \eqref{concb} is a Feller resolvent on $\cosu $. It is conservative if all $\mem_{i,j}$s are probability measures.  
\end{thm} 
\begin{proof} We start with the Hilbert equation for $\rlac, \lam >0$. Fix $\lam, \mu >0$ and $f\in \cosu$.  By definition \eqref{concb}, 
\begin{align*}\rmic \rlac f &=\rmidu (\rlac f) + \sum_{i\in \mc M} \sum_{j=1}^{\ka (i)} u_{i,j} (\rlac f,\mu) \ell_\mu^{i,j} \\
& = \rmidu \rladu f + \sum_{i\in \mc M} \sum_{j=1}^{\ka (i)} u_{i,j} (f,\lam) \rmidu \ell_\mu^{i,j} + \sum_{i\in \mc M} \sum_{j=1}^{\ka (i)} u_{i,j} (\rlac f,\mu) \ell_\mu^{i,j}.
 \end{align*}
Therefore, by the resolvent equation and the definition of an exit law,
\begin{align*} (\lam - \mu) \rmic \rlac f & = \rmidu f - \rladu f + \sum_{i\in \mc M} \sum_{j=1}^{\ka (i)} u_{i,j} (f,\lam) (\ell_\mu^{i,j} - \ell_\lam^{i,j}) \\ &\phantom{=}+  (\lam - \mu) \sum_{i\in \mc M} \sum_{j=1}^{\ka (i)} u_{i,j} (\rlac f,\mu) \ell_\mu^{i,j}. \end{align*}
This differs from $\rmic f - \rlac f$ by the sum
\[  \sum_{i\in \mc M} \sum_{j=1}^{\ka (i)} [u_{i,j}(f,\lam) - u_{i,j}(f,\mu) +   (\lam - \mu)  u_{i,j} (\rlac f,\mu)] \ell_\mu^{i,j}. \]
However, all summands here are zero, by Lemma \ref{cbnrl2}, and so we obtain $(\lam -\mu)\rmic \rlac f= \rmic f- \rlac f.$ Since $\lam,\mu$ and $f$ are arbitrary, this shows that $\rlac , \lam >0$ is a pseudoresolvent.  

Next, by Lemma \ref{cnp:lem1}, $u\ge 0$ whenever $f\ge 0$. It follows that the operators $\rlac f, \lam >0$ are non-negative. 
Moreover, by Lemma \ref{pdojs}, 
\[\rlac 1_{\suu} = \rladu 1_{\suu} +  \isumap \sum_{j=1}^{\ka(i)} u_{i,j} (1_{\suu},\lam) \ell_\lam^{i,j} \le \rladu 1_{\suu} +  \isumap \sum_{j=1}^{\ka(i)} \ell_\lam^{i,j} \]
and the right hand side coincides, by \eqref{brzegip} and \eqref{sumdel}, with $ \rladu 1_{\suu} + \dell = 1_{\suu}.$ Hence, 
\begin{equation}\label{rlasms} \lam \rlac 1_{\suu} \le 1_{\suu}, \end{equation}
 and if all $\mem_{i,j}$s are probability measures, the inequality above can be replaced by equality --- this is a consequence of the other part of Lemma \ref{pdojs}. 

We are thus left with showing that $\lil \lam \rlac f = f$. Since $\rlai, \lam >0$ are regular, we have $\lil \lam \rladu f =f$ and thus it suffices to show that 
\begin{equation}\label{lil} \lil  \isumap \sum_{j=1}^{\ka(i)} \left ( \int_{S} \lam \rlac f \ud \mem_{i,j} \right )\ell_\lam^{i,j} =0 \end{equation}
However, inequality  \eqref{rlasms} combined with non-negativity of the operators $\rlac, \lam >0$ shows that 
\[ \|\lam \rlac \| \le 1 , \qquad  \lam >0 .\]
Therefore, the norm of the sum in \eqref{lil} does not exceed 
\[ \|f\| \sum_{i\in \mc M} \sum_{j=1}^{\ka(i)} \|\ell_\lam^{i,j}\|. \]
This converges to zero, as $\lam \to \infty$, because, by Corollary \ref{cor1}, $\lil \|\ell_\lam^{i,j}\| =0$ for all $(i,j) \in \mc{IJ}$. 
\end{proof}


\subsection{In (\ref{concp}) Laplace transforms of exit laws need to be used}\label{must} 

As we have see above, the fact that $\ell^{i,j}_\lam, \lam >0$ is the Laplace transform of an exit law was crucial in proving the resolvent equation for $\rlac , \lam >0$.  In fact, the defining equation \eqref{owc:2} was used both in the proof of Theorem \ref{thm5} and in the proof of Lemma \ref{cbnrl2}, preceding the theorem. The following result may be thought of as a converse to Theorem \ref{thm5}; we prove namely that formula \eqref{concp} will not work, unless  $\ell^{i,j}_\lam$s are Laplace transforms of exit laws. 

Suppose thus that, as in Section \ref{setup}, we have $N$ Feller processes, of which $M$ first are not conservative, and that $\ell_\lam^{i,j} \in \cosi, (i,j)\in \mc{IJ} $ are given functions such that $0\le \ell_\lam^{i,j}\le 1_{S_i}$ and \eqref{concp} defines a resolvent in $C(\suu)$ for all measures $\mem_{i,j}$. We claim that then, for each $(i,j)\in \mc {IJ},$ 
\begin{equation}(\lam - \mu) \rmii \ell_\lam^{i,j} = \ell_\mu^{i,j} - \ell_\lam^{i,j}, \qquad \lam, \mu >0.\label{claimik} \end{equation}

To prove this, we fix $(i,j) \in \mc{IJ}$, choose a $k\not = i$, a point $x\in S_k$, and take all measures involved to be zero, except for $\mem_{i,j} $ which we define to be the Dirac measure at $x$. Then \eqref{concp} takes the form
\begin{equation*} \rlac f =  \rladu f  +  [\rlac f (x)] \ell_\lam^{i,j}.  \end{equation*}
Since $\ell^{i,j}_\lam $ is supported in $S_i$ and $i\not =k$, it follows that $\rlac f(x) = \rladu f(x)$, and the formula simplifies to 
\begin{equation*} \rlac f =  \rladu f  +  [\rladu f (x)] \ell_\lam^{i,j}.  \end{equation*}
(The resulting resolvent is related to the following concatenation of processes: a particle that moves about initially  in $S_i$ and eventually exits this region by the $j$th gate, starts all over again at the $x \in S_k$ and moves according to the law  
known in $S_k$, until the process is possibly no longer defined. However, if the initial position is not in $S_i$ or the process does not exit through the $j$th gate, after its lifetime is over the process is not continued at all.)

A short calculation yields
\begin{equation} \label{shorcal}(\lam - \mu) \rmic \rlac f = (\lam - \mu) \bigl [\rmi\rladu f + [\rladu f(x)] \rmii \ell_\lam^{i,j} + [\rmidu \rlac f(x)] \ell_\lam^{i,j}\bigr ].\end{equation}
Furthermore, since $\ell_\lam^{i,j}$ is supported in $S_i$, so is $\rladu \ell_\lam^{i,j}$, and hence the relation
\[ \rmidu \rlac f = \rmidu \rladu f + [ \rladu f(x) ] \rmidu \ell_\lam^{i,j} \] proves that $\rmidu \rlac f(x) = \rmidu \rladu f(x)$ (because $x\in S_k$ and $S_k$ is disjoint with $S_i$). Thus, the right-hand side of \eqref{shorcal} differs from $\rmic f - \rlac f$ by 
\[ [\rladu f(x)] \left ((\lam - \mu) \rmii \ell_\lam^{i,j} + \ell_\lam^{i,j} - \ell_\mu^{i,j} \right ),\]
$\rladu, \lam >0$ being  a pseudoresolvent. 

Since, $\rlac, \lam >0$ is a pseudoresolvent also, to complete the proof it suffices to find an $f \in \cosi$ such that $\rmidu f(x)$ is non-zero for all $\mu >0$. To this end, let $f$ be non-negative and such that $f(x) >0$. Since the paths of the process $X_k$ with generator $A_k$ are right-continuous on a set of probability $1$, if we restrict ourselves to the paths starting at $x$, we will have $f(X_k(t))>0$ for sufficiently small $t \in [0,\infty)$, and for other $t$ we will have $f(X_k(t))\ge 0$. It follows that $\int_0^\infty \e^{-\lam t} f(X_k(t)) \ud t >0$ for all $\lam >0.$ Therefore, by the Fubini Theorem, 
\[\rladu f(x)  = E_x   \int_0^\infty \e^{-\lam t} f(X_k(t)) \ud t >0, \qquad \lam >0, \] 
as desired. 
\section{Examples of regular Feller boundary}\label{examples} 
In this section we discuss three examples of processes with regular Feller exit boundary composed of finite number of gates (identified with excessive functions), as defined in Definition \ref{defn:fb}. These processes are Brownian motions on $[0,\infty]$, $[0,r]$ where $r>0$, and star-like graphs, respectively, with rather general boundary conditions describing the gates. We will be able to provide explicit formulae for exit laws and related excessive functions that feature in these examples: in the first of them, there is one gate, in the second --- two gates, and in the third --- $\ka \ge 2$ gates.
Moreover, this section exemplifies that fact that Theorem \ref{thm5} together with the result of Section \ref{must} can be instrumental in finding formulae for exit laws and excessive functions.     
\subsection{Example A: $S=[0,\infty]$; one exit law.} \label{exa1}
Consider the space $C[0,\infty]$ of continuous functions on $\R^+=[0,\infty)$ with limits at infinity. Let $A$ be the operator $Af= \frac 12 f''$ with domain composed of twice continuously differentiable functions with $f'' \in C[0,\infty]$, satisfying the boundary condition \index{boundary conditions} \index{condition!boundary} 
\begin{equation} af''(0) - b f'(0) + cf(0) =0 ,\label{bc} \end{equation}
where $a,b$ and $c$ are given non-negative constants with $a + b +c >0.$

It is well-known that, as long as $a+b>0$, $A$ generates a Feller semigroup on $C[0,\infty]$, and that the resolvent $\rla = \rez{A}, \lam >0$ of $A$ is given by (see e.g. \cite{knigaz}*{pp. 17-18})   
\begin{equation*}
\rla f (x) = C \e^{\sqrt {2\lam} x } + D \e^{-\sqrt {2\lam} x } - \sqrt{\frac 2\lam} \int_0^x \sinh \sqrt {2\lam} (x-y) f(y) \ud y, \qquad x\ge 0,\end{equation*}
for all $f\in C[0,\infty]$, where $C=C(\lam,f)$ and $D=D(\lam,f) $ are constants defined as follows:
\begin{align*} C = \frac 1{\sqrt {2\lam} } \int_0^\infty \e^{-\sqrt {2\lam} y } f(y) \ud y, \nonumber \mquad{ }
D =  \frac{(b\sqrt {2\lam} - 2a\lam - c)C + 2af(0)}{2a\lam + b \sqrt {2\lam} + c}. \end{align*}
For $f = 1_{S}$, the constants are $C= \frac 1{2\lam}$ and $D= \frac 1{2\lam} \frac {2a\lam + b\sqrt{2\lam} -c}{2a\lam + b\sqrt{2\lam}+ c}$. This renders 
\[ \lam \rla 1_{S} = 1_S - \frac c{2a\lam + b\sqrt{2\lam} +c} e_{\sqrt{2\lam}}, \]
where $e_{\sqrt{2\lam}}(x) = \e^{-\sqrt{2\lam}x }, x \ge 0.$ 
 
Therefore, as long as $c>0$, we have $\lam \rla 1_{S} < 1_{S}$, that is, the resolvent is not conservative. In this case, the Laplace transform of the lifetime of the related process is 
\begin{equation}\dell = 1_{S} - \lam \rla 1_{S} =  \frac c{2a\lam + b\sqrt{2\lam} +c} e_{\sqrt{2\lam}}. \label{exa1:1} \end{equation}
Moreover, as it is easy to see, the kernel of $A$ is trivial, and thus we have the first example of a regular Feller boundary with (at least) one exit. An intuitive description of the related process tells us furthermore that this process can exit the state-space only through $0$, and therefore, there is precisely one exit here. 

We note also that in this example, the excessive function for the exit law is $1_{[0,\infty]}$: this corresponds to the fact that the process considered here with probability $1$ eventually hits $0$, regardless of its starting point.

Two sub-cases are of interest here. First of all, if $b=0$, the origin is a so-called \emph{elementary exit} (comp.  \cite{fellera1}*{p. 3} or  \cite{knigaz}*{p. 19}).  
The related process is a standard Brownian motion on the right half axis which, however, after reaching the origin is captured there for an exponential time with parameter $\frac c{2a};$  after this time is over, the process is no longer defined. 

Formula \eqref{exa1:1} confirms this description: $\frac c{2a\lam +c}$ is the Laplace transform of the exponential time with parameter $\frac c{2a}$, whereas $e_{\sqrt{2\lam}} (x)$ is the Laplace transform of the time needed for the standard Brownian motion starting at $x>0$ to reach the origin --- see \cite{ito}*{ Eq. 5) p. 26}, \cite{karatzas}*{Eq. (8.6) p. 96} or the original \cite{levybook}*{pp. 221-223}. 
 
Secondly, in the case of $a=0$, the boundary condition \eqref{bc} describes the \emph{elastic Brownian motion}. This is a process in the right half axis which initially behaves as if it were a reflected Brownian motion. However, the time it spends at the origin is measured by means of the local time of P.~L\'evy; when an exponential time with parameter $\frac cb$ passes with respect to this local time, the process is left undefined. This description is again confirmed by \eqref{exa1:1}, because  $\frac c{b\sqrt{2\lam} +c}$ is the Laplace transform of the distribution of  exponential time with parameter $\frac cb$  with respect to the L\'evy local time --- see \cite{ito}*{p. 45, Eq. 1}, that is, of the time elastic Brownian motion needs to exit the space after reaching the origin for the first time. 

If $a=b=0\not =c$, we encounter the \emph{minimal} Brownian motion --- here the Laplace transform of the time needed for the process starting at $x>0$ to exit the space is $\e^{-\slam x}$, that is, the Laplace transform of the time needed to reach the origin for the first time. Strictly speaking, though, $0$ is not in the state-space of the minimal Brownian motion, and the related semigroup is defined in the subspace $C_0(0,\infty]$  of $C[0,\infty]$ of $f $ from the latter space which satisfy $f(0)=0$.

For $a>0$, the origin is more `sticky' than in the case of $a=0$, and so the time process spends at the boundary is longer; in the case of $c=0$ the process is known as \emph{sticky} or \emph{slowly reflecting} Brownian motion --- see
\cite{kostrykin2010brownian}, \cite{liggett}*{p. 127} or \cite{revuz}*{p. 421}. For non-zero $c$, $\frac c{2a\lam + b\sqrt{2\lam} +c} $ is the Laplace transform of the time needed for the related process to exit the space after reaching the boundary for the first time.

\subsection{Example B: $S=[0,r]$; two exits (with exceptions)}\label{exa2}
Let $S=[0,r]$ be a closed interval of length $r>0$, and let $A$ be the operator in $\coss$ given by $Af=\frac 12 f''$ on the domain composed of $f\in \coss$ that are twice continuously differentiable and such that
\begin{equation}\label{exa2:1} p f''(0) - (1-p) f'(0) + c f(0) =0 \mquad {and} q f''(r) + (1- q) f'(r) + d f(r) = 0,
 \end{equation} 
where $p,q\in [0,1]$ and $c,d\ge 0.$ 
To explain, the condition on the left is a new incarnation of \eqref{bc}: take $p=\frac a{a+b}$ and redefine $c$ as $\frac c{a+b}$. The condition on the right is of similar nature.

From this description it is clear that, as long as $c$ and $d$ are larger than $0$, there are two exits here. We will find the corresponding exit laws and excessive functions with the help of Theorem \ref{thm5} and the result of Section \ref{must}. Since  these excessive functions turn out to be continuous and add up to $1_{[0,r]}$, the process considered here has a Feller exit boundary.  

Before continuing, we note that a standard argument (see e.g. \cite{knigaz}*{p. 17}) can be used to show that $A$ satisfies the positive maximum principle and is densely defined. The calculations presented below prove also that it satisfies the range condition, and thus is a generator of a Feller semigroup. The same remark applies to $\aco$ introduced below. 

\subsubsection{Auxiliary concatenation}\label{aco}
To find exit laws for $A$, we think of $S$ as of $S_1$ in Theorem \ref{thm5} and add an auxiliary space $S_2$ composed of two points: $S_2 \coloneqq \{o_1,o_2\}$ ($o_1,o_2\not \in [0,r]$) and equip $S_2$ with discrete topology.  An auxiliary process in $S_2$ is trivial, that is, both points are absorbing: the process starting at one of these points stays there for ever. In other words, its generator is $A_2=0$. We concatenate these two processes by requiring that (a) when the process related to $A_1=A$ exists $[0,r]$ through the boundary point $0$, it jumps to $o_1$ to stay there for ever, and (b) when the process exists through $r$, it jumps to $o_2$ to stay there for every.  

The generator $\aco$ of the concatenated process in $\suu = [0,r]\cup \{o_1,o_2\} $ is defined as follows. Its domain is composed of $f$ which, as restricted to $[0,r]$, are twice continuously differentiable,  and satisfy the following transmission conditions:
\begin{align}\nonumber 
p f''(0) - (1-p) f'(0) + c f(0) &= c f(o_1), \\ \label{exa2:2} q f''(r) + (1- q) f'(r) + d f(r) &= d f(o_2).
 \end{align} 
Also, $\aco f (x)= \frac 12 f''(x), x\in [0,r],$ whereas $\aco f(o_1)=\aco f(o_2)=0$.

Next, we solve the resolvent equation for $\aco$: given $g \in C(S_1\dot\cup S_2)$ and $\lam >0$ we find an $f  \in \dom{\aco}$ such that $\lam f - \aco f = g.$ The part of this equation in $C(S_2)$ is very simple and we immediately realize that $\lam f(o_i) = g(o_i), i=1,2.$ 
To solve the other part we search for $f$ of the form
\begin{equation}\label{exa2:3}  f (x) = C \e^{\sdlam x} +   D \e^{-\sdlam x} + h_\lam  (x), \qquad x \in [0,r] \end{equation}
where $h_\lam(x)\coloneqq \frac 1{\sdlam} \int_0^r \e^{-\sdlam |x-y|} g(y) \ud y$,
such that transmission conditions 
\eqref{exa2:2} are satisfied. In fact, for reasons that will become clear in a moment, we will find a unique $f$ satisfying slightly more general conditions:
\begin{align}
\nonumber p f''(0) - (1-p) f'(0) + c f(0) &= \widetilde c f(o_1), \\  q f''(r) + (1- q) f'(r) + d f(r) &= \widetilde d f(o_2),\label{exa2:4}
 \end{align} 
where $0 \le \widetilde c\le c$ and $0\le \widetilde d \le d.$

\subsubsection{A linear system for $C$ and $D$}\label{alsf} We note that \eqref{exa2:4} is fulfilled iff $C$ and $D$ of \eqref{exa2:3} 
are solutions to the system of linear equations:
\begin{align*}
a_{1,1} C + a_{1,2} D & = b_1 + \widetilde c f(o_1), \\
a_{2,1} C + a_{2,2} D & = b_2 + \widetilde d f(o_2),\end{align*}
where 
\begin{align}
a_{1,i} & \coloneqq 2\lam p  + (-1)^i \slam  (1-p) + c ,\label{exa2:4a}\\
a_{2,i} & \coloneqq \e^{(-1)^{i+1} \slam r} (2\lam q  - (-1)^i \slam  (1-q) + d), \qquad i=1,2,\nonumber \end{align}
and  (since $h_\lam'(0) = \slam h_\lam (0) $ and $h_\lam'(r) = -\slam h_\lam (r)$)
\begin{align}
b_1 &\coloneqq 2p g(0) - (2\lam p - \slam (1-p) + c) h_\lam (0), \nonumber \\
b_2 &\coloneqq 2q g(r) - (2\lam q - \slam (1-q) + d) h_\lam (r).\label{exa2:4b} \end{align} 
It is clear that $|a_{1,1}|\le a_{1,2}$ (with equality for $p=1$), $|a_{2,2}|< a_{2,1}$ and that neither of $a_{1,2}$ and $a_{2,1}$ is zero. Hence, the main determinant $\wyr$ of the system is negative, and the system has a unique solution. 

\subsubsection{Two exit laws} It follows that 
\begin{align} \nonumber  
\rlac g (x)& = \frac 1 \wyr \begin{vmatrix}b_1& a_{1,2} \\ b_2 & a_{2,2}  \end{vmatrix} \e^{\slam x} + \frac 1\wyr  \begin{vmatrix}a_{1,1}& b_{1} \\ a_{2,1} & b_{2}  \end{vmatrix} \e^{-\slam x} + h_\lam (x)  \\ &\phantom{=}+ \frac 1\wyr  \begin{vmatrix}\widetilde c f(o_1) & a_{1,2} \\ \widetilde d f(o_2) & a_{2,2}  \end{vmatrix} \e^{\slam x}  + \frac 1\wyr  \begin{vmatrix}a_{1,1}& \widetilde c f(o_1) \\ a_{2,1} & \widetilde d f(o_2)  \end{vmatrix} \e^{-\slam x}. \label{exa2:5}
 \end{align} 
The first three terms here correspond to the case of $\widetilde c = \widetilde d =0$, that is, provide the formula for the resolvent of $A$. The remaining two may be written as
\[ \frac {\widetilde c} \wyr (a_{2,2} \e^{\slam x} - a_{2,1} \e^{-\slam x})  f(o_1)+   \frac {\widetilde d} \wyr  (a_{1,1} \e^{-\slam x} - a_{1,2} \e^{\slam x})  f(o_2).  \] 
Since 
\[ f(o_i) = \frac{g(o_i)}{\lam } = \rlai g(o_i) = \rlac g(o_i) , \qquad i=1,2,\]
formula \eqref{exa2:5} is a particular case of  \eqref{concp} with $N=2, M=1, \ka (1) = 2,\mem_{1,1}=\frac {\widetilde c} c\delta_{o_1} $ and $\mem_{1,2}=\frac {\widetilde d} d\delta_{o_2}$: 
\[ \rlac g = \rladu g + {\textstyle\frac {\widetilde c}c} [\rlac g(o_1)] \ell_\lam^1 +  {\textstyle \frac {\widetilde d}d} [\rlac g(o_2)] \ell_\lam^2  \]
where 
\begin{align}
\ell_\lam^1 (x) & \coloneqq \frac c\wyr  (a_{2,2} \e^{\slam x} - a_{2,1} \e^{-\slam x}),\nonumber \\
\ell_\lam^2 (x) & \coloneqq \frac d \wyr  (a_{1,1} \e^{-\slam x} - a_{1,2} \e^{\slam x}), \qquad x \in [0,r]. \label{exa2:5a}
\end{align}
Above, we write $\ell_\lam^1$ and $\ell_\lam^2$ instead of $\ell_\lam^{1,1}$ and $\ell_\lam^{1,2}$ for simplicity and, at the same time, to comply with notation of Section \ref{elae}.

The argument of Section \ref{must} is in force here and we conclude that $\ell_\lam^j, j=1,2$ are exit laws for $A$. We also note that the resolvent of $A$ can be expressed in terms of these laws as follows: 
\begin{equation}\label{exa2:5b} \rez{A} g = \frac{b_1}c \ell_\lam^1 + \frac{b_2}d \ell_\lam^2 + h_\lam . \end{equation}   

\subsubsection{Two excessive functions} 
To find the excessive functions corresponding to exit laws \eqref{exa2:5a}, we let 
\[ \nad^j (x) \coloneqq \lim_{\lam \to 0+} \ell_\lam^j (x), \qquad j =1,2, \]
and write 
\begin{equation}\label{exa2:5c} \ell_\lam^1 = c \frac {\wyr_1}{\wyr} \mquad{ and }  \ell_\lam^2 = d \frac {\wyr_2}{\wyr}\end{equation}
where
\begin{equation} \wyr_1 \coloneqq \begin{vmatrix} \e^{\slam x} & \e^{-\slam x}\\
a_{2,1} & a_{2,2} \end{vmatrix} \mquad { and }  \wyr_2 \coloneqq  \begin{vmatrix} a_{1,1} & a_{1,2} \\[4pt] \e^{\slam x} & \e^{-\slam x} \end{vmatrix}. \label{exa2:6}  \end{equation}
Then formula \eqref{aff:1} in Appendix shows that 
\begin{align}
\nad^1 (x) &= c \frac{d(r-x) + 1-q}{ cdr + (1-q) c + (1-p) d}, \nonumber \\
\nad^2 (x) &= d \frac{cx + 1-p}{ cdr + (1-q) c + (1-p) d}, \qquad x \in [0,r]. \label{exa2:7} \end{align} 
For $p=q=1$ (this is the case of elementary exits from both interval's endpoints) these quantities are familiar probabilities that a Brownian motion starting at $x\in [0,r]$ hits $0$ before hitting $r$, and vice versa, respectively. In view of the well-known relationship between excessive functions and exit probabilities (see \cite{doobc}*{p. 565}), in general these are apparently the probabilities of exiting through $0$ and $r$, conditional on the process generated by $A$ starting at $x$. Notably, 
\[ \nad^1 + \nad^2 =  1_{[0,r]},\]
and both $\nad^j, j=1,2$ are continuous functions: with probability one the process eventually exits either through the left-end or the right-end of the interval. It is also quite easy to see that the kernel of $A$ is trivial. 
Thus, as heralded before, the process considered here has Feller exit boundary with (at least) two exits --- it is the intuitive description of the related process that tells us that there are precisely two exits here.

\subsubsection{A particular case}\label{apc} For $p=q=1$ (two elementary exits) formulae for excessive functions simplify as follows: 
\begin{align*} \ell_\lam^1 (x) &=  \frac {r-x} r \left ( \frac r{r-x} \frac{\sinh \sdlam (r-x)}{\sinh \sdlam r} \right ) \frac c{2\lam +c }, \qquad x \in [0,r), \\
\ell_\lam^2 (x) &= \frac xr \left ( \frac r{x} \frac{\sinh \sdlam x}{\sinh \sdlam r} \right ) \frac d{2\lam +d}, \qquad x\in (0,r], \end{align*}
and these relations have a clear interpretation. To wit, the first factor in the formula for $\ell_\lam^1(x)$ is the probability that a standard Brownian starting at $x\in [0,r]$ will reach $0$ before reaching $r$. The second is the Laplace transform of the time needed for the standard Brownian motion to reach $0$, conditional on reaching it before reaching $r$ --- see \cite{karatzas}*{p. 100}. The third is the Laplace transform of the exponential time the process generated by $A$ spends at $0$ before exiting the state-space. The formula for $\ell_\lam^2$ is interpreted similarly. 

\subsubsection{Another particular case} Another case of interest is $p=1$ and $q=0$ (elementary exit on the left, elastic boundary at the right). Here, 
\begin{align*} \ell_\lam^1 (x) &=  \frac{\slam \cosh \sdlam (r-x)}{d\sinh \sdlam r + \slam \cosh \slam r} \frac c{2\lam +c }, 
\\
\ell_\lam^2 (x) &=  \frac{\sinh \sdlam x}{d\sinh \sdlam r + \slam \cosh \slam r}, \qquad x\in [0,r]. \end{align*}
It is interesting to note how the change in the nature of the boundary at $r$ influences the exit at the left. In the new scenario, a particle reaching $r$ for the first time is not captured there for an exponential time as in Section \ref{apc}, but is reflected and can thus exit through $0$. Hence, the probability of exiting through $0$, equaling 
$\nad^1(x) = \frac {1+d(r-x)}{1+dr} $ is larger than that obtained in Section \ref{apc}.

\subsubsection{The `degenerate' case: one exit law}\label{tdg} If either $c=0$ or $d=0$ (but not both), the related process has one exit law. For instance, we consider the case of elastic boundary at $0$ and reflecting boundary at $r$, that is, the case where $p=q=d=0$ whereas $c\not =0.$ Then, $\ell_\lam^2 (x) $ reduces to zero, reflecting the fact that the process cannot exit $[0,r]$ through the interval's right end. Moreover, 
\[ \ell_\lam^1 (x) = \frac {\cosh \slam (x-r)}{ \cosh \slam r} \frac {c \cosh \slam r}{c \cosh \slam r + \slam \sinh \slam r}, \qquad x\in [0,r]. \]
The first factor here is the Laplace transform of the distribution of the time needed for the Brownian motion starting at $x\in (0,r)$ and reflected at $x=r$ to reach $x=0$ for the first time (see Appendix \ref{ajeden}). Also, the second factor is completely monotone  (see Appendix \ref{adwa}) and its value for $\lam =0$ is $1$, proving that it is the Laplace transform of a Borel probability measure on $[0,\infty)$. The second factor thus  should be interpreted as the Laplace transform of the time needed to exit $[0,r]$ after reaching $x=0$ for the first time.   

Since $\lilz \ell_\lam ^1(x) = 1$, this exit law corresponds to the excessive function $1_{[0,r]}$, that is,  $\ell_\lam^1 $ coincides with $\dell .$

\newcommand{\stargraphs}[3]{\begin{tikzpicture}[scale=0.2]
      \node[circle,fill=black,inner sep=0pt,minimum size=4pt] at (360:0mm) (center) {};
    \foreach \n in {1,...,#1}{
        \node [circle,fill=black,inner sep=0pt,minimum size=3pt] at ({\n*360/#1}:#2cm) (n\n) {};
        \draw [dashed] (center)--(n\n);} 
    \foreach \n in {#3,...,#1}{
    \draw (center) -- (n\n);}
\end{tikzpicture}}
\newcommand{\stargraph}[2]{\begin{tikzpicture}
      \node[circle,fill=black,inner sep=0pt,minimum size=2pt] at (360:0mm) (center) {};
    \foreach \n in {1,...,#1}{
        \node [circle,fill=blue,inner sep=0pt,minimum size=3pt] at ({\n*360/#1}:#2cm) (n\n) {};
        \draw (center)--(n\n);} 
\end{tikzpicture}}
\newcommand{\stargraphn}[2]{\begin{tikzpicture}
  \node[circle,fill=white,inner sep=0pt,minimum size=3pt] at (360:0mm) (center) {};
    \foreach \n in {1,...,#1}{
        \node [circle,fill=black,inner sep=0pt,minimum size=2pt] at ({\n*360/#1}:#2cm) (n\n) {};
        \draw (center)--(n\n);} 
\end{tikzpicture}}

\subsubsection{A black sheep: representation in terms of excessive functions is not unique}\label{sheep} 
The case of $p=q=1$, $d=0$ and $c>0$ is somewhat exceptional (together with the analogous case of $p=q=1$, $c=0$ and $d>0$). Here, we have again only one exit, but representation in terms of excessive functions is not unique. 

First of all, as in the previous section, in this scenario $\ell_\lam ^2(x)$ reduces to zero. Moreover, 
\[ \ell_\lam^1(x) = \frac {c\sinh\slam (r-x)}{(2\lam +c) \sinh \slam r }, \qquad x \in [0,r]. \]
Since $\lilz \ell_\lam^1 (x) =1$ for all $x\in [0,r]$, we see that, again, $\ell_\lam^1 =\dell .$ However, for any $\alpha \in \R$, the function $f(x) =\alpha x , x \in [0,r]$ belongs to the kernel of $A$. As a result,  for any $\alpha \in [0,\frac 1r]$,
\[ \ell_\lam^1 = \nad^{\alpha} - \lam \rla \nad^\alpha\]
where 
\[ 0\le \nad^\alpha (x) \coloneqq 1- \alpha x\le 1, \qquad x\in [0,r]. \]

\subsubsection{A remarkable property of exit laws}\label{remarkable}
Coming back to the general case and $c,d>0$ we note the following remarkable properties of exit laws. To begin with, we note 
that (derivatives with respect to $x$)
\begin{equation} \label{exa2:8} (\ell_\lam^j)'' = 2\lam \ell_\lam^j , \qquad j=1,2.\end{equation}
Next, we  calculate:
\begin{align*} p(\ell_\lam^1)'' (0) &- (1-p)(\ell_\lam^1)'(0) + c \ell_\lam^1 (0) \\ &= \frac {c}{\wyr}((2\lam p +c)(a_{2,2} - a_{2,1}) - (1-p)\slam (a_{2,2} - a_{2,1})) \\&= \frac {c}{\wyr}(a_{1,1}a_{2,2}-a_{2,1} a_{1,2}) = c. \end{align*}
At the same time
\begin{align*} p(\ell_\lam^2)'' (0) &- (1-p)(\ell_\lam^2)'(0) + c \ell_\lam^1 (0) = \frac {c}{\wyr}(a_{1,1}a_{1,2}-a_{1,2} a_{1,1}) = 0. \end{align*}
Introducing $F_1:C^2[0,r]\to \R$ given by $F_1f\coloneqq p f''(0) -(1-p) f'(0) + cf(0)$ we summarize the `algebraic' nature of exits laws at the interval's left end as follows:
\begin{equation} \label{exa2:9}
F_1 \ell_\lam^j = \delta_{1,j} c, \qquad j =1,2,\end{equation}
where $\delta_{1,j}$ is the Kronecker delta. 
There is also a counterpart of these relations at the right end: 
\begin{align} 
F_2 \ell_\lam^j = \delta_{2,j} d, \qquad j =1,2\label{exa2:10}
\end{align}
where $F_2 f\coloneqq q f''(r) +(1-q) f'(r) + df(r).$ 
It is worth noting that, \eqref{exa2:8}--\eqref{exa2:10} in fact determine $\ell_\lam^1$ and $\ell_\lam^2$ uniquely.

We also remark that quantities $b_1$ and $b_2$ of \eqref{exa2:4b} are equal to $-F_1(h_\lam)$ and $-F_2(h_\lam)$, respectively. This together with $\lam h_\lam - \frac 12h_\lam = g$ and \eqref{exa2:8}--\eqref{exa2:10} yields an easy proof of \eqref{exa2:5b}. 

We hasten to add that the properties of exit laws discussed above are not characteristic to this particular example merely. For the process  considered in Section \ref{exa1} we similarly have $(\dell)'' = 2\lam \dell$ and $F\dell =c $ where $Ff \coloneqq af''(0)- bf'(0) + c.$ Likewise, for $\ell_\lam^1, \lam >0$ of Section \ref{tdg} we have $(\ell_\lam^1)''=2\lam \ell_\lam^1$ and $-(\ell_\lam^1)' (0) + c\ell_\lam^1 (0) = c.$ 

\subsection{Example C: Walsh-type Brownian motion on a star-like graph}\label{exa3}
Let $r>0$ and a natural number $\ka \ge 2$ be given. Moreover, let $S$ be the union of $\ka$ copies of the interval $[0,r]$, with all left ends identified. In other words, $S$ is the star-like graph $K_{1,\ka}$. Continuous functions on $S$ can be identified with continuous functions $f$ on \[ \widetilde S\coloneqq \bigcup_{j\in \kad}\{j\} \times [0,r], \qquad \text{where } \kad \coloneqq \{1,\dots, \ka\}, \] 
that satisfy $f(j,0)=f(k,0), j,k\in \kad $; their common value at $0$ will be denoted $f(0)$. It will also be handy to write $f\in \coss$ as a sum of $f_j\in C[0,r]$ such that $f_j(0)=f_k(0), j,k\in \kad :$
\[ f = \dot{\sum_{j\in \kad}} f_j.\]  
Let $\beta\in [0,1)$ and $\alpha_j, j \in \kad$ be non-negative numbers such that \[ \beta + \sum_{j\in \kad} \alpha_j  =1 .\] Also, let $q \in [0,1]$ and $d>0$ be given. We define an operator $A$ in $\coss$ as follows. Its domain is composed of $f\in \coss $ such that 
\begin{itemize} 
\item [(a)] $f \in C^2(S);$ that is, for each $j \in \kad $, $f_j$ is twice continuously differentiable on $[0,r]$ and 
\begin{equation}\label{exa3:1} f'' \coloneqq \dot{\sum_{j\in \kad}} f_j'' \end{equation}
belongs to $\coss$,
\item [(b)] for each $j\in \kad$, 
\[ qf_j''(r) + (1-q) f'_j (r) + df_j(r) = 0,\]
\item [(c)] and
\[ \beta f''(0) = \sum_{j\in \kad} \alpha_j  f_j'(0)\]
(note that $f''(0)$ is well-defined by (a), even though $f'_j(0)$ depends on~$j$). 
\end{itemize}
Moreover, we agree that
\[ Af = {\textstyle \frac 12} f''. \] 

Boundary condition (b) is of the same form as the second relation in \eqref{exa2:2}, and thus says that each `outer' node of the graph is an exit point for the process generated by $A$. Transmission condition (c) describes the law governing the process at the graph's center. If $\beta$ is zero, this condition is known to describe the Skew Brownian Motion
of J. Walsh --- see \cite{lejayskew}*{p. 45 eq. (57)}, consult also \cite{yor97}*{p. 107}  and \cite{manyor}*{pp. 115-117}. 
In this process, a particle reaching the center of the graph chooses to continue its further motion on $j$th vertex with `probability' $\alpha_j$ (although not quite) --- see the already cited \cite{lejayskew} for a more precise description and many equivalent constructions of this process;  in particular, transmission condition (c) with $\beta=0$ can be obtained in the limit procedure of Friedlin and Wentzell's averaging principle \cite{fwbook} --- see \cite{fw}*{Thm. 5.1}, comp. \cite{emergence}*{Eq. (3.1)}. In general, when $\beta$ is non-zero, the center is more `sticky' and so the process stays at this point `longer', but no probability mass is lost there (if $\beta$ were allowed to be $1$, the 
graph's center would be an absorbing point). In its more general form (i.e., with $\beta\not =0$), condition (c) appears also e.g. in \cite{kostrykin2012}*{Eq. (2.5)}, with yet one more term, describing loss of probability mass.  

From this intuitive description it is clear that there are precisely $\ka$ exits here. We will argue that $A$ is a Feller generator, and calculate the corresponding Laplace transforms of exit laws and the related excessive  functions. 

\subsubsection{The positive maximum principle for $A$}\label{pmp}
It is rather easy to be persuaded that $A$ is densely defined. Turning to the question of the positive maximum principle, we note that if $f\in \dom{A}$ attains its positive maximum at a point $x$ different from the nodes of $S$, we obviously have $Af(x)=\frac 12 f''(x)\le 0$. The already recalled argument from \cite{knigaz}*{p. 17} shows also that $Af(x)\le 0$ if $x$ is one of the `outer' nodes. Finally, if $x$ is the graph's center, we clearly have $f'_j(0)\le 0$ for all $j\in \kad$. Hence, if $\beta >0$, we have $Af(0)= \frac 12 f''(0)= \frac 1{2\beta}  \sum_{j\in \kad} \alpha_j  f_j'(0)\le 0.$ If $\beta=0$, (c) requires that $f_j'(0)=0$ as long as $\alpha_j>0$. Since there is at least one $\alpha_j>0$, for at least one $j$ we have $f_j'(0)=0$, and it follows that for this $j$, $f''_j(0)\le 0$. But $f$ belongs to $\dom{A}$, and so $f_j''(0)$ does not depend on $j$ and we have $f''(0) = f_j''(0), j\in \kad$, proving that $Af(0)\le 0$.   
To check that $A$ is a Feller generator, it suffices therefore to prove that it satisfies the range condition. This will be a by-product of the analysis presented in the next section.

\subsubsection{The resolvent equation for $A$ and $\aco$}
We proceed as in Section \ref{aco}: we concatenate the process generated by $A$ with the trivial process on 
\[ S_2 \coloneqq \{o_1,\dots , o_\ka\} \]
(for $o_j\not \in S, j \in \kad $), and from an explicit form of the resolvent of the concatenated process read off the exit laws we search for. As a by-product, our calculations (with $\widetilde d_j =0, j\in \kad$, see below) provide the proof that $A$ satisfies the range condition and is thus a Feller generator. 

Given additionally $ \widetilde d_j\in [0, d], j \in \kad $,  we define an operator $\aco$ in $\cosu$, where $\suu \coloneqq S_1 \dot{\cup} S_2$ and $S_1\coloneqq S$, as follows. Its domain is composed of   
\[ f= \dot{\sum_{j\in \kad}} f_j \dot{+} f_{|S_2}\in \cosu  \] 
 such that 
 \begin{itemize}
\item [(a)] $f_{|S} \in C^2(S)$, 
\item [(b)] for each $j\in \kad$, 
\[ qf_j''(r) + (1-q) f'_j (r) + df_j(r) = \widetilde d_j f(o_j),\]
\item [(c)] and
\[ \beta f''(0) = \sum_{j\in \kad} \alpha_j  f_j'(0).\]
\end{itemize}
For such $f$, we define $\aco f = \frac 12 f_{|S}''$ (in particular, $A(f_{|S_2})= 0$). 

\newcommand{\hlai}{h_{\lam,i}}
\newcommand{\hlaj}{h_{\lam,j}}
Our claim is that $\aco$ is a Feller generator in $\cosu$. The related process, if started in $S$, obeys the rules dictated by $A$; when it exits $S$ through the $j$th gate, it jumps to  $o_j$ to stay there for ever.

Arguing as in Section \ref{pmp}, we check that $\aco$ (is densely defined and)
satisfies the positive maximum principle (this is where condition $0\le \widetilde d_j \le d$ is used). Hence, we are left with showing that $\aco $ satisfies the range condition.  

To this end, given $\lam >0$ and $g\in \cosu$ we look for $f \in \dom{\aco}$ satisfying $\lam f - \aco f =g$ 
of the form: 
\begin{align} \nonumber f(o_j) &= \lam^{-1} g(o_j), \\
\label{exa3:2}  f_j(x) & = C_j \e^{\sdlam x} +   D_j \e^{-\sdlam x} + \hlaj  (x), \qquad x \in [0,r], j \in \kad,  \end{align}
where 
\begin{equation}\hlaj (x)\coloneqq \frac{2\beta g(0)}{\slam(\beta - 1)}\sinh\slam x - \sqrt{\frac 2{\lam}} \int_0^x \sinh \sdlam (x-y) g_j(y) \ud y, \label{exa3:2a}\end{equation}
for $ j \in \kad,  x\in [0,r]$, and the constants $C_j$ and $D_j$ are to be found. We also introduce functionals $F,F_j$  given by 
\begin{align*} F f &\coloneqq \beta f''(0) - \sum_{j\in \kad} \alpha_j f_j'(0) \\
F_j f &\coloneqq q f_j''(r) + (1-q) f_j'(r) + d f_j(r) ,\qquad j \in \kad, \end{align*}
on the common domain of $f$ such that $f_{|S}\in C^2(S)$.
Hence, $f\in \dom{\aco}$ iff $f_{|S} \in C^2(S)$, $Ff=0$ and $F_jf  =\widetilde d_j f(o_j), j \in \kad$.

To continue, $ \dot\sum_{j\in \kad}f_j\in C(\widetilde S)$ belongs to $\coss$ iff 
\begin{equation}\label{exa3:3} C_j +D_j  \quad \text{ does not depend on }j \in \kad . \end{equation}
Since $\lam f_j - \frac 12 f_j''= g_j, j \in \kad$, this is also a necessary and sufficient condition for $f''$ to belong to $\coss$. If  \eqref{exa3:3} holds, the quantity featured in  it is $f(0)$, and we obtain
\begin{equation}\label{exa3:3a} C_j = f(0) - D_j, \qquad j \in \kad .\end{equation}
Next, as in Section \ref{alsf} we check that $F_j f = \widetilde d_j f(o_j)$ iff 
\begin{align*}
a_{2,1} C_j + a_{2,2} D_j  =  \widetilde d_j f(o_j)- F_j h_\lam, \qquad j \in \kad \end{align*}
where $a_{2,i}, i=1,2$ are defined in \eqref{exa2:4a}, and \begin{equation}\label{exa3:3b} h_\lam \coloneqq \dot \sum_{j\in \kad} \hlaj .\end{equation}
These relations together yield 
\begin{equation}\label{exa3:4} (a_{2,2} - a_{2,1}) D_j =   \widetilde d_j f(o_j) - F_jh_\lam - a_{2,1}f(0), \qquad j \in \kad.  \end{equation}
Finally, (c) of the definition of $\dom{\aco}$ holds iff 
\begin{equation}\label{exa3:5} \beta (2\lam f(0) - 2g(0)) = \sum_{j\in \kad} \slam \alpha_j (C_j-D_j).\end{equation}
To solve the system  \eqref{exa3:3a}, \eqref{exa3:4}--\eqref{exa3:5} for $C_j$ and $D_j, j\in \kad$, we note first that the right hand-side in \eqref{exa3:5} equals 
\begin{align*}    
\slam \sum_{j\in \kad}  \alpha_j (f(0) -2 D_j) = \slam (1-\beta) f(0) - 2\slam  \sum_{j\in \kad} \alpha_j D_j, 
\end{align*}
and that, by  \eqref{exa3:4}, 
\[ (a_{2,2} - a_{2,1})\sum_{j\in \kad} \alpha_j D_j= \sum_{j \in \kad}  \alpha_j(\widetilde d_j f(o_j)-F_jh_\lam) - a_{2,1}(1-\beta) f(0).\]
Therefore,
\begin{equation}\label{exa3:5aa} f(0) = \frac {2\beta}{m_\lam}g(0) + \frac {2\slam}{m_\lam (a_{2,1}-a_{2,2})} \sum_{j\in \kad} \alpha_j (\widetilde d_jf(o_j)- F_jh_\lam), \end{equation}
where 
\[ m_\lam\coloneqq 2\beta \lam + \slam (1-\beta) \frac{a_{2,1}+a_{2,2}}{a_{2,1}- a_{2,2}}.\]  
Once $f(0)$ is known, $D_j$s can be calculated from  \eqref{exa3:4} and then $C_j$s can be calculated from  \eqref{exa3:3a}.
The fact that the so-derived $C_j$s and $D_j$s solve the system \eqref{exa3:3}--\eqref{exa3:5} with $\widetilde d_j = d, j \in \kad $ shows that \eqref{exa3:2}
 defines a solution (which is necessarily unique) to the resolvent equation  for $\aco$. The same calculation with all $\widetilde d_j$s equal to $0$ proves that $A$ satisfies the range condition also.  
\subsubsection{Exit laws for $A$}
Fix $j\in \kad.$ The $f_j$ of \eqref{exa3:2} with constants calculated from \eqref{exa3:3a}, \eqref{exa3:4}--\eqref{exa3:5}  differs from the $f_j$ with constants calculated from this system with  $\widetilde d_k =0, k \in \kad$ by 
\[ \widetilde d_j \frac{ 2\sinh\slam x}{a_{2,1} - a_{2,2}}f(o_j) + \frac {2\slam (a_{2,1}\e^{-\slam x} - a_{2,2}\e^{\slam x})}{m_\lam (a_{2,1}-a_{2,2})^2} \sum_{k \in \kad} \alpha_k \widetilde d_k f(o_k). \]
Hence, the resolvent of $\aco$ has the form 
\[ \rlac g  = \rladu g + \sum_{j\in \kad } {\textstyle \frac{\widetilde d_j}{d}} [\rlac g(o_j)] \ell_{\lam}^j,\]
where $\ell_\lam^j\in \coss$ is defined by its parts $(\ell_\lam^i)_k \in C[0,r], k \in \kad $ as follows
\begin{equation} (\ell_\lam^j)_k (x) = \frac{ 2d \sinh\slam x}{a_{2,1} - a_{2,2}} \delta_{j,k} + \frac {2d\slam (a_{2,1}\e^{-\slam x} - a_{2,2}\e^{\slam x})}{m_\lam (a_{2,1}-a_{2,2})^2} \alpha_j ,\label{exa3:5a} \end{equation}
and $\delta_{j,k}$ is the Kronecker delta. 

As in Section \ref{remarkable}, we note that 
\[ (\ell_\lam^j)'' = 2\lam \ell_\lam^j, \qquad j \in \kad ,\]
and
\begin{align*} F_k \ell_\lam^j  &= \frac{d (a_{2,1}- a_{2,2})}{a_{2,1}- a_{2,2}} \delta_{j,k} + \frac {2d\slam (a_{2,2}a_{2,1} - a_{2,2}a_{2,1})}{m_\lam (a_{2,1}- a_{2,2})^2} \alpha_k \\ &= d\delta_{j,k}, \qquad   \qquad   \qquad   \qquad   \qquad   \qquad    \qquad    \qquad   \qquad   \qquad   \qquad   j,k\in \kad. \end{align*}
A bit longer calculation (see Appendix \ref{puff}) shows that 
\begin{equation} F\ell_\lam^j = 0, \qquad j \in \kad .\label{exa3:6a} \end{equation}
Together with the observation that  $F h_\lam =0$ ($h_\lam $ was defined in \eqref{exa3:3b}) these relations lead to the following formula for $\rez{A}$ in terms of exit laws (comp. \eqref{exa2:5b}): 
\begin{equation}\label{exa3:7} \rez{A} g = h_\lam  - d^{-1}\sum_{j\in \kad } (F_jh_\lam) \ell_\lam^j  .\end{equation}

\subsubsection{Excessive functions for $A$}
Using the formula in Appendix \ref{aff}, it can be seen that
\begin{align*} 
\lilz \frac{a_{2,1}\e^{-\slam x} - a_{2,2}\e^{\slam x}} {2\slam } & = d(r-x) + 1-q,\\ 
\lilz {m_\lam} &= \frac{d(1-\beta)}{rd + 1-q}.
\end{align*}
Therefore, as long as $\beta <1$, the excessive function corresponding to the exit through  the $j$th outer node  is $ \nad^j \coloneqq \dot{\sum}_{k\in \kad}(\nad^j)_k$ where  
\begin{equation}\label{exa3:6} (\nad^j)_k (x) \coloneqq  \lilz (\ell_\lam^j)_k (x) = \frac {dx}{rd + 1-q} \delta_{j,k} + \frac{d(r-x)+ 1-q}{rd + 1-q} \frac {\alpha_j}{1-\beta}, \end{equation}
 for $x \in [0,r]$. It is clear that $(\nad^j)_k(0)$ does not depend on $k$, proving that $\nad^j$ belongs to $\coss.$ 
Since, as we shall see in the next section, in the case of $\beta<1$, the kernel of $A$ is trivial, we conclude that 
\[ \ell_\lam^j = \nad^j - \lam \rla \nad^j, \qquad j \in \kad .\]
To interpret \eqref{exa3:6}, we note that the quantity $\frac {dx}{rd + 1-q}$ coincides with the right-hand side in the second equation in \eqref{exa2:7} with $p=1$. Since in the case of $p=1$ the left interval's end in the  example of Section \ref{exa2} is absorbing, this quantity is the probability that a particle performing a Brownian motion on $[0,r]$ with right-end boundary described by the second equation in \eqref{exa2:1} exits through this boundary without touching $0$ in the meantime. Since the nature of the boundary at the $i$th node is the same as in the process just described (see (b) in the definition of $\dom{A}$), $\frac {dx}{rd + 1-q}$ is the probability that a Walsh-type Brownian motion starting at the point $x$ at the $j$th node will exit through this node's end without ever going through the graph's center. On the other hand, $\frac{d(r-x)+ 1-q}{rd + 1-q}$ is the probability that a particle starting from this $x $ will reach the graph's center before exiting from the state-space (necessarily through the $i$th node). Finally, $\frac {\alpha_k}{1-\beta}$ is the probability that a process starting from the graph's center will exit the graph through its $k$th node. 

It is an interesting feature of this example that even though all the outer nodes, that is, all the exits, have the same `local' nature (being described by a boundary condition with the same coefficients), the related exit laws differ --- see the formulae for $\ell_\lam^j, \lam >0$ and $\nad^j, j \in \kad. $  The reason for that lies in the graph's center --- because of that point the process has a tendency to sojourn in some vertices more often than in other vertices, and thus exits through one node  more often than through the other. 

\subsubsection{Triviality of the kernel} Finally, we show that, as long as $\beta\in [0,1)$, the kernel of $A$ is trivial so that, in particular, the representation of exit laws in terms of excessive functions is unique. This together with the considerations presented above establishes that the process related to $A$ has a regular Feller exit boundary with $\ka\ge 2$ exits.

To this end, we note that if $f \in \Ker A$, then $f_j(x) = m_jx + b, x\in [0,r], j \in \kad$ for certain real numbers $b$ and $m_j, j\in \kad$, because $f_j(0) $ is to be independent of $j$. Next, condition (b) tells us that we should have \begin{equation}\label{exa3:8} (1-q)m_j + d(m_j r +b)=0.\end{equation} It follows that $m_j$ does not depend on $j$ either, and thus can be denoted by $m$. Also, by (c), we need to have $0=\sum_{j\in \kad} \alpha_j m = (1-\beta)m$, implying $m=0$. Since, by \eqref{exa3:8}, this forces $b=0$ and, consequently, $f=0$, we are done. 

For $\beta=1$, the kernel of $A$ is not trivial. It contains all functions of the form $f=\dot\sum_{j\in \kad} f_j$ where
\[ f_j(x) \coloneqq mx - d^{-1}(1-q+dm)r, \qquad x \in [0,1], j \in \kad .\]

\section{A sample generation theorem} \label{asgt}

In the foregoing sections we described concatenated stochastic processes by two apparently different means. In Section \ref{cbnr} we used resolvents, whereas in Section \ref{examples} each example was commenced by an account in terms of  generator and corresponding boundary and transmission conditions. Although is was more or less obvious from the calculations presented in the examples that these are two descriptions of the same stochastic process, we never proved this formally. This is what we aim to do in this section. 

To this end, we come back to generation theorems for diffusions on graphs of the type considered in \cite{bobmor,nagrafach,gregosiewicz,gregosiewiczsem,banfal2,gregnew,binz}. To recall, in these papers, it is imagined that graph's vertices are semi-permeable membranes. Therefore, it is convenient to think of the graph as the union of, say, $N$, disjoint intervals: 
\[ \suu = \dot {\bigcup_{i\in \mc N}} \{i\} \times [0,r_i], \]
where, as before, $\mc N\coloneqq \{1,\dots, N\}$ and $r_i, i\in \mc N$ are positive numbers (edge's lengths).  
Each interval's (i.e., edge's) end is considered as different even from the ends of adjacent edges, since it is seen as lying on `this side' of the membrane, as opposed to `the other side' of the membrane. Moreover, the membrane's permeability 
depends both on the edge from which a diffusing particle filters and on the edge to which the diffusing particle filters.  

A more precise definition involves $N$ quadruples of parameters
\[ (p_i,c_i,q_i,d_i), i \in \mc N \]  
such that $p_i,q_i \in [0,1]$ and $c_i,d_i>0$, and $2N$ sub-probability measures $\mem_{i,j}, i\in \mc N$ and $j=1,2$ on $S$.
The domain of the generator $A$ of the process considered is composed of $f=\dot\sum_{i\in \mc N} f_i$ such that $f_i$ is twice continuously differentiable on $[0,r_i]$, and the following boundary and transmission conditions are satisfied (cf. \eqref{exa2:2})
\begin{align} \nonumber  p_i f_i''(0) - (1-p_i)f_i'(0) + c_i f_i(0) &= c_i \int f \ud \mem_{i,1},\\
q_i f_i''(r_i) + (1-q_i)f_i'(r_i) + d_i f_i(r_i) &= d_i \int f \ud \mem_{i,2}, \qquad i \in \mc N \label{asgt:1}
 \end{align}
(these are particular examples of non-local Feller--Wentzel conditions).
For such $f$, 
\[ Af \coloneqq {\textstyle \frac 12} f'' \coloneqq {\textstyle \frac 12} \dot {\sum_{i\in \mc N}} f_i''. \]
In \cite{nagrafach} (and in \cite{bobmor}, where $N=3$), $p_i$s are zeros; the case of $p_i\not =0$ is considered in \cite{gregnew}; the measures $\mem_{i,1}$ and $\mem_{i,2}$ are convex combinations of Dirac measures at the endpoints of edges that are adjacent to $(i,0)$ and $(i,r_i),$ respectively. In \cite{banfal2} and \cite{marjeta,binz} more general operators and transmission conditions are discussed, including those without a clear probabilistic interpretation. Moreover, analysis is carried out also in spaces other than the space of continuous functions. The requirement that $\mem_{i,j}$s are convex combinations of Dirac measures at the ends of adjacent edges is natural and means that after filtering through the end of the $i$th vertex, a particle starts at one of the ends of the adjacent edges, choosing an edge with probability dictated by the convex combination. In what follows, however, we will work with general $\mem_{i,j}$s: all we assume is that they are sub-probability measures. 

The latter assumption is necessary for the conclusion that the operator $A$ defined above satisfies the positive maximum principle. More precisely, it guarantees that if a positive maximum  of a function $f\in \dom{A}$ is attained at one of the vertices, the value of $Af$ at this vertex is not positive (cf. \cite{knigaz}*{p. 17} or our Section \ref{pmp}). Since it is rather clear that $A$ is densely defined, to prove that $A$ is a Feller generator, one needs to check the range condition. As mentioned in the introduction, this is by far the hardest task; in the previous papers, various methods, rarely with clear probabilistic meaning, were used to accomplish it, including Greiner-like perturbations (see \cite{greiner} and \cite{banasiak,banfal2,3Dlayers}, and references given therein, comp. also \cite{arkuk2})), passing to isomorphic semigroups (see \cite{bobmor,nagrafach}), and lifting techniques (see   \cite{banfal1}*{Section 3.2}) that are often used in singular perturbation theory involving less standard boundary conditions --- see, for example, \cite{banasiak}*{Lemma 2.3}, or \cite{banbob2}*{pp. 230--232}. 

Meanwhile, the range condition is a direct consequence of our Theorem \ref{gen}, a bit further down, which, in turn, is a corollary to Theorem \ref{thm5}.    
 
Before presenting the theorem, we need some notation. To begin with, for $i\in \mc N$, let $A_i$ be the operator in $C[0,r_i]$ with domain composed of $C^2[0,r_i]$ functions, that is, of twice continuously differentiable functions, satisfying the boundary conditions 
\begin{align} \nonumber  p_i f''(0) - (1-p_i)f'(0) + c_i f(0) &=0,\\
q_i f''(r_i) + (1-q_i)f'(r_i) + c_i f(r_i) &= 0, \label{asgt:2}
 \end{align}
 and defined by $A_i f= \frac 12 f''.$ In Section \ref{exa2} we have established that there are two exit laws, say, $\ell_{\lam}^{i,j}, j=1,2$ for $A_i$; they are given by \eqref{exa2:5a} and \eqref{exa2:4a} with $p,q,c$ and $d$ replaced by $p_i,q_i,c_i$ and $d_i$, respectively.  
 
 Also, let $\rladu, \lam >0$ be the resolvent of the disjoint union of the related processes (see \eqref{du}). Finally, let $\rlac , \lam >0$ defined by
 \[ \rlac g = \rladu g + \sum_{i\in \mc N}\sum_{j=1,2} u_{i,j} \ell_\lam^{i,j}, \qquad g \in \coss,\lam >0, \]
be the resolvent of the process concatenated by means of the measures $\mem_{i,j}$ featured in \eqref{asgt:1}; this means that $u_{i,j}$s are unique solutions to the counterpart of the system \eqref{mapka},
and in fact, see \eqref{bttd:1},
\begin{equation}\label{asgt:3} u_{i,j} = \int \rlac g \ud \mem_{i,j}, \qquad i\in \mc N, j=1,2. \end{equation}

\begin{thm}\label{gen} For any $\lam >0$ and $g\in \coss$ there is a unique $f\in \dom{A}$ solving the resolvent equation $\lam f - Af =g$. This $f$ is 
\[ f = \rlac g. \]  
 \end{thm}
\begin{proof}Let $C^2(\suu)$ be the subspace of $\cosu $ of $f=\dot\sum_{i\in \mc N} f_i$ such that $f_i\in C^2[0,r_i]$ (that is, $f_i$ is twice continuously differentiable), and let $\Delta:C^2(\suu)\to \cosu$ be given by $\Delta f \coloneqq \frac 12 f'' \coloneqq \dot \sum_{i\in \mc N} f_i''.$  Also, let $F_{i,j}, i\in \mc N, j=1,2$ be the following functionals on $C^2(\suu)$:
\begin{align*}
F_{i,1} f &=   p_i f_i''(0) - (1-p_i)f_i'(0) + c_i f_i(0),\\ 
F_{i,2} f &= q_i f''_i(r_i) + (1-q_i)f'_i(r_i) + d_i f(r_i), \qquad i \in \mc N.  
\end{align*}
Since the kernel of $F_{i,j}$ includes all functions that are equal to zero on the $i$th edge and because of the definition of $A_k$,  we see that  if $f = \dot\sum_{k\in \mc N} f_k$ is such that $f_k \in \dom{A_k}, k \in \mc N$, we have $F_{i,j} f = 0$.  In particular, $F_{i,j} \rladu g =0$ for all $g\in \cosu.$ 

Moreover, see \eqref{exa2:8}--\eqref{exa2:10},
\begin{equation} \label{asgt:4} \Delta \ell_\lam^{i,j} = \lam \ell_\lam^{i,j} , \qquad i\in \mc N, j=1,2,\end{equation}
and 
\begin{align} 
F_{i,1} \ell_\lam^{i,j} = \delta_{1,j} c_i \mquad { and }
F_{i,2} \ell_\lam^{i,j} = \delta_{2,j} d_i, \qquad j =1,2.\label{asgt:5}
\end{align}
Therefore, 
\( F_{i,1} \rlac g = F_{i,1} \rladu + c_i u_{i,1} = c_i u_{i,1},\)  
and, similarly, $F_{i,2} = d_i u_{i,2}$. When combined with \eqref{asgt:3}, this means that 
\[ F_{i,1} \rlac = c_i \int \rlac \ud \mem_{i,1} \mquad { and }  F_{i,2} \rlac = d_i \int \rlac \ud \mem_{i,2},\] 
that is, by \eqref{asgt:1}, that $\rlac g $ belongs to $\dom{A}$. Moreover, since $A$ is a restriction of $\Delta$ and so is the generator 
of the process of disjoint union, 
\begin{equation}\label{asgt:6} (\lam - A) \rlac g = (\lam - \Delta) \rladu g  = g \end{equation} 
where we also used \eqref{asgt:4}. 

We have shown that $\rlac g$ is a solution to $\lam f - A f=g$. To complete the proof, we note that, since $A$ satisfies the positive maximum 
principle, it is dissipative by  \cite{ethier}*{Lemma 2.1, p. 165}: for any $f \in \dom{A}$, $\|\lam f - Af\|\ge \|f\|.$ Therefore, the solution is unique. \end{proof}

\begin{cor}The operator $A$ coincides with the generator $\aco$ of the process related to $\rlac, \lam >0$.\end{cor}
\begin{proof} We have proved that $\dom{\aco}\subset \dom{A}$. Moreover, by \eqref{asgt:6}, $\lam \rlac g  - \aco \rlac g = g= \lam \rlac g - A\rlac g$, proving that $A\rlac g = \aco \rlac g, g \in \cosu$. In other words, $\aco $ is a restriction of $A$ to $\dom{\aco}$. On the other hand, both operators are generators of Feller semigroups, and thus are dissipative. This is impossible unless $\dom{A} = \dom{\aco}$ (see \cite{ethier}*{Proposition 4.1, p. 21}), as desired. 
 \end{proof}

Before closing this section, we note that the argument presented in the proof of Theorem \ref{gen} applies to 
more general scenarios. First of all, for some $i$s, either $c_i$ or $d_i$ may be zero, as long as we avoid the 
case described in Section \ref{sheep}. In such a scenario there  is  only one exit for the process related to $A_i$ 
and one of transmission conditions in \eqref{asgt:1} should be changed appropriately. Nevertheless, the argument is the same because the remaining exit law has the properties listed at the end of Section \ref{remarkable} (these properties allow writing counterparts of \eqref{asgt:4}--\eqref{asgt:5}). Likewise, 
some of the intervals can be replaced by half-lines and $A_i$s can be replaced by generators of the type discussed in Section \ref{exa1}. Finally, each interval can be replaced by the star-like graph of Section \ref{exa3}  and the corresponding generator can be replaced by the generator of Walsh-type Brownian Motion. In the last scenario, we would perhaps face more than 
two transmission conditions for each component of $\suu$, but again the argument would be the same, modulo natural changes. The same remark applies to more general graphs like those in \cite{kostrykin2012}. 

Secondly, even though we restrict our applications to stochastic processes on graphs (which are our main motivation), it is worth stressing that  the results of Section \ref{cbnr} apply also to processes with exit laws that are not described by boundary and transmission conditions.

Finally, a comment is also in order on the remarkable properties of exit laws \eqref{asgt:4}--\eqref{asgt:5} that played such a crucial role in the proof of Theorem \ref{gen}. Functions with such or similar properties are typically used as handy tools in lifting techniques (see the references cited above). Thus, it appears that the analysis presented in this paper is another reflection of the fact that classical analysis has its probabilistic counterpart, with the latter providing intuitions that are not easily visible in the former (cf. \cite{doobc}).

\section{Convergence of semigroups} \label{s:cos}
Our last section (Section \ref{aap}) is devoted to an applications of the results obtained in the first part of the paper: we prove an averaging principle for fast processes, of the kind considered before in \cites{banfal1,bobmor,nagrafach,zmarkusem}. These results hinge on the theory of convergence of semigroups which is recalled briefly here.

The classical Trotter--Kato Theorem (see e.g. \cites{goldstein,pazy}) says that strongly continuous equibounded semigroups \( \sem{B_\eps}, \eps \in (0,1]\) in a Banach space $E$ converge as $\eps \to 0$ to a strongly continuous semigroup $\sem{B}$, that is,  
\begin{equation}\label{alt:1} \grae \e^{tB_\eps } f= \e^{tB}f, \qquad t \ge 0, f \in E, \end{equation}
iff 
\[ \grae \rez{B_\eps} f = \rez{B} f, \qquad  f \in E, \]
for some/all $\lam >0$; moreover, then the limit \eqref{alt:1} is uniform in $t$ in compact subsets of $[0,\infty)$. 
In other words, such regular convergence of semigroups is completely characterized (see also \cites{kniga,knigaz,ethier} for the Sova--Kurtz version \cites{kurtz,sova} of this characterization). 

However, in the theory of singular perturbations and in the particular example we are studying in this paper the limit semigroup is strongly continuous only on a subspace of $E$: we are facing a limit theorem of the form 
\begin{equation}\label{alt:2}  \grae \e^{tB_\eps} f = \e^{tB} P f , \qquad t >0, f \in E\end{equation}
where \sem{B} is a strongly continuous semigroup on a subspace $E_0$ of $E$ and $P$ is a projection on $E_0$ (in the sense that $P^2=P$ and $Pf =f , f \in E_0$). Needless to say, in this case the classical theory does not work and, in particular, condition 
\begin{equation}\label{alt:3} \grae  \rez{B_\eps} f = \rez{B} Pf, \qquad  f \in E, \end{equation}
for all/some  $\lam >0$ is necessary but not sufficient for \eqref{alt:2} (see \cite{deg} or \cite{knigaz}).

Condition \eqref{alt:3} can imply \eqref{alt:2} provided that the semigroups involved possess additional regularity properties (like, for example, uniform holomorphicity --- see e.g. \cite{knigaz}*{Chapters 31 and 41} for details). A different set of conditions guaranteeing that \eqref{alt:3} implies \eqref{alt:2} has been given by T. G. Kurtz \cites{ethier,kurtz,kurtzper,kurtzapp}. While Kurtz's singular convergence theorem is usually expressed in terms of the so-called extended limit of generators, for our subsequent analysis the following resolvent-version will be more practical. This result can be easily deduced e.g. from combined Lemma 7.1 and Theorem 42.2 in \cite{knigaz}.    

\begin{thm} \label{tkurtza} Suppose $B_\eps, \eps \in (0,1]$ are generators of strongly continuous equibounded semigroups. Suppose also that  for some $\lam >0$
\begin{equation}\label{alt:4} \grae \rez{\eps B_\eps} = \rez{A_0} \end{equation}
where $A_0$ is the generator of a strongly continuous semigroup \sem{A_0} such that 
\[ Pf\coloneqq \grat \e^{tA_0} f , \qquad f \in E\]
 exists. Then condition \eqref{alt:3}  (for some $\lam >0$, with the same $P$) implies \eqref{alt:2}, and the limit is uniform in $t$ in compact subsets of $(0,\infty)$; for $f\in E_0$ the limit is uniform in $t$ in compact subsets of $[0,\infty).$ 
\end{thm}


\section{An averaging principle}\label{aap}

\subsection{Asymptotically splittable processes} 
Before we begin our main subject, let us first think what happens when a non-honest process is accelerated. In other words, how multiplying the generator $A$ of a non-honest process by $\eps^{-1}$, where $\eps \ll 1$  influences the process. It is quite easy to see that a $\nad$ is excessive for $\eps^{-1}A$ iff it is excessive for $A$. Also, a $\widetilde \ell_\lam, \lam >0$ is an exit law for $\eps^{-1}A$ iff it is of the form $\tilde \ell_\lam = \ell_{\eps \lam}$ where $\ell_\lam, \lam >0$ is an exit law for $A$. Since, by the representation of \eqref{owc:dod}, $\lim_{\eps \to 0} \ell_{\eps \lam} (x) = m_x ([0,\infty))$ we conclude that in the limit of accelerated processes the probabilities of exiting through various exits do not change, but the process's lifetime is $0$. ($\lam \mapsto  m_x ([0,\infty))$ is the Laplace transform of the Dirac measure at $0$ multiplied by $  m_x ([0,\infty))$.)

A more interesting limit is obtained if while the process accelerates, the speed with which probability mass escapes from the state-space is tuned so that the fluxes through the boundary's exit points are more or less constant. These considerations lead to the following definition. 

\newcommand{\sfv}{\mathsf v}
\newcommand{\sfV}{\mathsf V}
\newcommand{\vr}{\varrho}

Let $ \sfv \in \R^\ka$ be a non-zero vector and let $A_{\eps \sfv}, \eps \in [0,1]$ be Feller generators in the space $\coss$ of continuous functions on a compact space $S$. We think of all these generators as describing `one stochastic process' with regular Feller boundary and $ \ka $ exit points, and of $\eps \sfv$ as a vector of parameters that describe the speeds with which probability mass escapes the state-space through exit points. In particular, we assume that $A_{\mathsf 0}$ is a conservative Feller process. Moreover, to each $\eps$  there correspond $\ka$ excessive functions describing exits, which we denote 
\[ \nad_{\eps}^j, \qquad j \in \kad (=\{1,\dots, \ka\}). \] 
\begin{defn} \label{asik} \emph{We will say that $A_{\sfv}$ is a generator of an asymptotically splittable process if the following conditions are satisfied. 
\begin{itemize} 
\item [(a) ] There are constants $\vr^j \in C(S)$ such that $\grae \nad_{\eps}^j = \vr^j 1_S, j \in \kad$ (necessarily $\vr^j\ge 0$ and $\sum_{j\in \kad} \vr^j =1$).  
\item [(b) ] $\grae \rez{A_{\eps \sfv}} f = \rez{A_{\mathsf 0}}f, f \in C(S)$. 
\item [(c) ] $A_{\mathsf 0} $ is a conservative Feller generator  and $\grat \e^{tA_{\mathsf 0}} f = Pf, f \in C(S)$ where $P$ is a projection of $\coss $ onto the subspace 
$\cosf \subset \coss$ of functions that are constant on $S.$ 
\item [(d) ] There is a $\gamma >0$ such that $\grae \rez{\eps^{-1}A_{\eps \sfv}} f = \frac 1{\lam + \gamma} Pf, f \in \coss$. 
\end{itemize}}
 \end{defn}

We note that in view of representation of exit laws in terms of excessive functions, if (d) holds, condition (a) implies convergence of corresponding exit laws. Denoting these laws by $\ell_{\lam,\eps}^j, \lam >0, \eps \in (0,1]$, we obtain namely   
\begin{equation}\grae \ell_{\lam,\eps}^j = \frac{\rho_j\gamma }{\lam + \gamma}1_S,\qquad \lam >0, j \in \kad.\label{asp:1} \end{equation} On the other hand, convergence of exit laws, because of formulae like \eqref{exa2:5b} and \eqref{exa3:7}, often leads to (d). We also note that one or several of $\rho^j$s in condition (a) can be zero --- some exits can be lost in approximation. All these points are exemplified below.

\subsection{Examples of asymptotically splittable processes}\label{eoas}

In this section we collect a number of typical examples of asymptotically splittable processes.  

\begin{ex} In Example \ref{ex0} of Section \ref{owc}, $\sfv$ is one-dimensional --- there is a single parameter, namely $\alpha$, that describes the speed with which probability mass escapes from $[0,1]$. Fixing $\alpha>0$ and denoting by $A_{\alpha}$ the operator $A$ from this example corresponding to parameter $\alpha$ we see what follows. 

(a) There is one exit law for each $A_{\eps \alpha}, 0< \eps \le 1$ corresponding to the same excessive function $\nad_\eps = 1_{[0,1]}$. 

(b) $\grae \rez{A_{\eps \alpha}} f= \rez{A_0}f, f \in C[0,1]$ by \eqref{exa0:1}.  

(c) $A_0$ is a conservative Feller generator and $\grat \e^{tA_0} f = f(1)1_{[0,1]}$. 

(d) Since 
\[ \rez{\eps^{-1}A_{\eps \alpha}} f(x) = \eps \int_x^1\e^{\eps \lam (x-y)} f(y) \ud y + \e^{\eps \lam (x-1)} \frac {f(1)}{\lam + \alpha},\] 
we have $\grae  \rez{\eps^{-1}A_{\eps \alpha}} f = \frac 1{\lam +\alpha}  f(1)1_{[0,1]}.$ 

It follows that $A_{\alpha}$ where $\alpha >0$ generates an asymptotically splittable process: here $\ka =1$, $\vr=1,$ $Pf= f(1)1_{[0,1]}$ and $\gamma = \alpha $. 
\end{ex}
 
\begin{ex}\label{ex:4} In Example \ref{ex01}, $A_\alpha$ also generates an asymptotically splittable process as long as we assume that $\grat \e^{tA} f = Pf, f \in \coss $ for a certain projection $P$ on $\cosf$. Indeed: 

(a) As before, the excessive functions do not depend on $\eps\in (0,1]$ and coincide with $1_S.$ 

(b) We clearly have $\grae \rez{A_{\eps a}} f= \grae (\lam +\eps \alpha - A)^{-1} f = \rez{A_0}f , f\in C(S)$ (where $A_0=A$). 

(c) Condition (c) is fulfilled by  assumption. 

(d) By the same assumption, $\lilz \lam \rez{A}f=Pf, f \in \coss$, and therefore \( \rez{\eps^{-1}A_{\eps \alpha}}f =\frac {\eps (\lam +\alpha)}{\lam +\alpha} (\eps (\lam +\alpha) - A)^{-1}f \) converges to $\frac 1{\lam +\alpha} Pf$, as $\eps \to 0$. In other words, again $\kappa=1, \vr =1$ and $\gamma = \alpha.$ \end{ex}


\begin{ex} \label{ex:5} In example of Section \ref{exa2}, we think of $p$ and $q$ as fixed and of $c$ and $d$ as measuring the speed with which probability mass escapes through the left and right ends of the interval $[0,r]$. In other words, $\sfv = (c,d)\in \R^2$.  To see that, as long as we do not have $p=q=1$, this is another example of an asymptotically splittable process we note what follows. 

(a)~There are two excessive functions corresponding to $A_{\eps \sfv}$; they are given by \eqref{exa2:7} with $c$ and $d$ replaced by $\eps c$ and $\eps d$, respectively. As a result, $\grae \nad^1_\eps = \frac {c(1-q)}{c(1-q) + d(1-p)}1_{[0,r]}$ and 
$\grae \nad^2_\eps = \frac {d(1-p)}{c(1-q) + d(1-p)}1_{[0,r]}$. 

(b) The quantities $a_{i,j}$ defined in \eqref{exa2:4a} with $c$ and $d$ replaced by $\eps c$ and $\eps d$ converge as $\eps \to 0$ to the same quantities with $c=d=0$, and the corresponding determinant is again negative. A similar remark concerns $b_1$ and $b_2$ of \eqref{exa2:4b}. It follows that the resolvent of $A_{\eps \sfv}$ which, to recall, is given by the first three terms in \eqref{exa2:5}, converges strongly to the resolvent of $A_{\mathsf 0}$. 

(c) $A_{\mathsf 0}$ is a conservative Feller process (because $1_{[0,r]}$ belongs to $\dom{A_{\mathsf 0}}$ and $A_{\mathsf 0} 1_{[0,r]} =0$). Moreover, arguing as in \cite{knigaz}*{Chapter 32}, that is, using the theory of asymptotic behavior of semigroups (see \cite{arendtsur,abhn,engel,nagel,vanneerven}, compare e.g. \cite{lasrud} and \cite{emelyanov}) it can be proved that there is a projection $P$ on $\cosf$ such that 
\[ \grat \e^{tA_{\mathsf 0}}g = Pg, \qquad g \in \coss . \]
Once existence of this limit is established, calculations presented in Section \ref{alsf} allow proving that 
\[ Pg \coloneqq \left (\frac {p(1-q)}M g(0) + \frac{q(1-p)}M g(r) + \frac{(1-p)(1-q)}M \int_0^r g(y) \ud y \right ) 1_{[0,r]},\]
where \begin{equation}\label{eoas:1} M \coloneqq p(1-q)+ q(1-p) + (1-p)(1-q)r\end{equation} is not zero by assumption (comp. \cite{3Dlayers}*{Theorem 4} where the particular case of $q=0$ and $r=1$ is covered).  To wit, if $c=d=0$ in \eqref{exa2:4a} and \eqref{exa2:4b}, then (see Appendix \ref{aff})
\begin{align*}
& \lilz \frac{\wyr}{(2\lam)^{\frac 32}} = - 2M \\
&\lilz b_1 = \beta_1 \coloneqq 2pg(0) + (1-p) \int_0^r g(y) \ud y, \\
& \lilz b_2 = \beta_2 \coloneqq 2qg(r) + (1-q) \int_0^r g(y) \ud y, \end{align*}
whereas
\begin{align*} \lilz \frac{a_{1,2}}{\slam} &= - \lilz \frac{a_{1,1}}{\slam} = 1-p, \\ \lilz \frac{a_{2,1}}{\slam} &= - \lilz \frac{a_{2,2}}{\slam}= 1-q.  \end{align*}
Moreover, we note that the limit $\grat \e^{tA_{\mathsf 0}} g$, if it exists, coincides with \( \lilz \lam \rez{A_{\mathsf 0}}g\) and, by \eqref{exa2:5}, 
 regardless of the choice of $x\in [0,r]$, $\lilz \lam \rez{A_{\mathsf 0}}g(x) $ equals 
 \begin{align*}
\lilz \frac{\lam \slam}{\wyr} & \begin{vmatrix} b_1 & { \frac{a_{1,2}}{\slam}} \\[5pt] b_2 & { \frac{a_{2,2}}{\slam}} \end{vmatrix} + \lilz \frac{\lam \slam}{\wyr} \begin{vmatrix} { \frac{a_{1,1}}{\slam}} & b_1 \\[5pt] { \frac{a_{2,1}}{\slam}} & b_2 \end{vmatrix} \\
& = \frac{-1}{2M} \begin{vmatrix} \beta_1 & 1-p \\ \beta_2 & q-1 \end{vmatrix} + \frac{-1}M \begin{vmatrix}  p-1 & \beta_1 \\ 1-q & \beta_2 \end{vmatrix}  = \frac{1}{M} \begin{vmatrix} \beta_1 & p-1 \\ \beta_2 & 1-q \end{vmatrix}. 
\end{align*}
Since this is just another expression for $Pg(x)$, our claim is proven. 

(d) If all occurrences of $\lam, g, c$ and $d$ in \eqref{exa2:4b} are replaced by $\eps \lam, \eps g, \eps c$ and $\eps d$, respectively, 
\begin{align*}
\grae \frac{b_1}{\eps } &= 2p g(0) + (1-p) \int_0^r g(x) \ud x ,\\
\grae \frac{b_2}{\eps } & = 2q g(r) + (1-q) \int_0^r g(x) \ud x.\end{align*}
When combined with \eqref{exa2:5c} and \eqref{aff:2} this shows 
\begin{align}
\grae \ell_{\lam,\eps}^1 & = \frac{c(1-q)}{2\lam M  + c(1-q) + d(1-p)} 1_{[0,r]}, \nonumber \\
\grae \ell_{\lam,\eps}^1 & = \frac{d(1-p)}{2\lam M  + c(1-q) + d(1-p)} 1_{[0,r]}. \label{eoas:2} \end{align}
Therefore, by \eqref{exa2:5b}, 
\[ \grae (\lam - \eps^{-1}A_{\eps \sfv})^{-1} = \grae (\eps \lam - A_{\eps \sfv})^{-1}\eps g = \frac 1{\lam +\gamma} Pg,\]
 where $\gamma \coloneqq \frac{c(1-q) + d(1-p)}{2M}.$ This completes the proof. 
 
Additionally, we remark that $\vr^1= \frac {c(1-q)}{c(1-q) + d(1-p)}$ becomes zero for $q=1$, and similarly $\vr^2 =\frac {d(1-p)}{c(1-q) + d(1-p)}$ is zero for $p=1$. This means that a process that has two exits can asymptotically 
lose one of them. 

In the case of $p=q=1$, the process of Example \ref{exa2} is not asymptotically splittable because in this scenario excessive functions, even though they do not depend on $\eps$, do depend on $x$.  
 \end{ex}   

\begin{ex} \label{ex:6}
As remarked in \ref{tdg}, in the example of Section \ref{exa2} with $d=0$ there is only one exit. Nevertheless, as long as $q\not =1$, the related process is still asymptotically splittable. We omit the details, since the calculations are similar to those presented in Example \ref{ex:5} (and are based on \eqref{exa2:5}). We merely note that, not surprisingly in view of Example \ref{ex:5}, in the case under study $\gamma = \frac{c(1-q)}{2M}.$ 
\end{ex} 

\begin{ex}\label{ex:7} We will argue that the process of Section \ref{exa3} is asymptotically splittable provided that $q\not =1$ and $\beta \not =1.$ Here, $\sfv$ is again one-dimensional and coincides with $d$. Parameters $q,\beta$ and $\alpha_j$s are treated as fixed.   

Convergence of excessive functions is clear: replacing $d$ with $\eps d$ in \eqref{exa3:6} and letting $\eps \to 0$, we obtain $\grae \phi_\eps^j = \frac{\alpha_j}{1-\beta}1_S.$ Point (b) of the definition is also immediate since all the quantities involved depend on $d$ continuously, and denominators obtained in the limit as $d\to 0$ differ from zero.    

Condition (c) is proved as follows. It can be argued that the semigroup generated by $A_0$ is holomorphic, irreducible and compact;  hence, there is a projection $P$ such that $\grat \e^{tA_0} g =Pg$ (we omit the details because they would lead much outside of the scope of the paper, comp. the works cited in Example \ref{ex:5}). However, a $Pg$ obtained as such a limit needs to belong to the kernel of $A_0$, and it is easy to check that the latter is composed of constant functions. It follows that $P$ is a projection on $\cosf$, and thus has the form $Pg= (F_P g)1_{S}$, where $F_P$ is a functional on $C(S)$. Moreover, since $Pg $ coincides with $\lilz \lam \rez{A_0} g$, the functional $F_P$ can be obtained as $\lilz \lam f(0)$ for the $f(0)$ defined in \eqref{exa3:5aa} with $d=0$ and $\widetilde d_j=0, j\in \kad$.

On the other hand, as long as $d=0$, by \eqref{aff:3} and the second part of \eqref{aff:1} (with $x=0$), $\lilz \frac{m_\lam}\lam (1-q) = 2 M$, where 
(cf. \eqref{eoas:1})
\begin{equation}\label{eoas:4} M \coloneqq \beta (1-q) + (1-\beta)q+ (1-\beta)(1-q)r. \end{equation}
Moreover, 
\[ \lilz F_j h_\lam =  -2 \left (qg_j(r) + (1-q)\int_0^r g_j(x) \ud x + \frac{\beta(1-q)}{1-\beta}g(0) \right ).\]
Hence, 
\begin{equation}\label{eoas:3} F_Pg  \coloneqq M^{-1}\beta(1-q)g(0) + M^{-1} \sum_{j\in \kad}\alpha_j\left ( qg_j(r) + (1-q)\int_0^r g_j(x) \ud x \right ). \end{equation}

To prove condition (d), we note that the first term in \eqref{exa3:5a} with $d$ and $\lam $ replaced by $\eps d$ and $\eps \lam$, respectively, converges to zero. Moreover, by the second formula in \eqref{aff:2}, the limit of the second term coincides with $\grae \frac{\eps d\alpha_j}{m_\lam (1-q)}.$ Also, \eqref{aff:3} and the already invoked part of formula \eqref{aff:2} show that 
$\grae \frac{m_\lam}\eps (1-q)= 2\lam M + (1-\beta) d$, where $M$ is defined in \eqref{eoas:4}. Hence, 
\[ \grae \ell_{\lam,\eps}^j = \frac {d\alpha_j}{2\lam M + (1-\beta) d}.\]

Furthermore, replacing $g$ and $\lam$ in the definitions \eqref{exa3:2a} and \eqref{exa3:3b} by $\eps g$ and $\eps \lam$ we obtain 
\[ \grae \frac{F_jh_\lam}{\eps } = -2 \left (qg_j(r) + (1-q)\int_0^r g_j(x) \ud x + \frac{\beta(1-q)}{1-\beta}g(0) \right ).\]
By \eqref{exa3:7}, this shows that 
\[ \grae \rez{\eps^{-1}A_{\eps d}} g = \frac 1{\lam + \gamma}Pg \]
with $\gamma\coloneqq \frac {1-\beta}{2M}$, where $Pg =(F_P g)1_S$ and $F_P$ is defined in \eqref{eoas:3}. 
\end{ex}

\subsection{An averaging principle} \label{ns:aap}

If $A_{\sfv}$ describes an asymptotically splittable process, conditions (b)--(d) of the definition tell us, by the Kurtz convergence theorem, that 
\[ \grae \e^{\eps^{-1}A_{\eps \sfv}t }f = \e^{-\gamma t} Pf, \qquad f \in \coss .\]
Since $P$ is a projection on the space of functions that are constant on $S$, this means in particular that as $\eps \to 0$ all points of $S$ are lumped together to form one single point of the state-space of the limit process. Furthermore, the factor $\e^{-\gamma t}$ says that after an exponential time, say, $T$, spent at this point, the limit process leaves the space and is no longer defined. In other words, the exponentially distributed time $T$ with parameter $\gamma$ is the lifetime of the limit process. Condition (a) is more specific about this scenario: it says that a particle escaping the single-point state-space does that through the $j$th gate with probability $\vr^j$: there are independent exponential random variables $T_j, j \in \kad$ with $\Pr (T_j\ge t) = \e^{-\vr^j  \gamma t }$, each representing waiting time at one of the $\ka$ gates. At $T\coloneqq \min_{j\in \kad} T_j$ (which is exponentially distributed with parameter $\gamma$)  the limit process leaves the state-space, and if $T = T_j$ (which happens with probability $\vr^j$) it does that through the $j$th gate --- functions $\lam \mapsto \frac{\vr^j \gamma}{\lam +\gamma }$ of \eqref{asp:1} are exit laws for the resolvent $\lam \mapsto \frac 1{\lam +\gamma}.$

It is clear from the description given above that asymptotically splittable processes are akin to Markov chains.  
Our main theorem in this section makes this connection more explicit (see Figure \ref{gurka}).  Namely, it says that if we concatenate $N$ asymptotically splittable processes then, by accelerating them while keeping the fluxes of probability mass through the boundaries approximately constant, we obtain a Markov chain. If initially we have $N$ processes to concatenate, defined in spaces $C(S_i), i \in \mc N$, then 
the state-space of the limit Markov chain is composed of $N$ points --- its $i$th  point is obtained by lumping together all the points in $S_i$. Moreover, the intensities of jumps in the limit chain depend on the measures $\mem_{i,j}$ that, to recall, describe starting points of the concatenated process after it exits $S_i$ through the $j$th gate --- see \eqref{mequ} further down.  

\begin{center}
\begin{figure}
\includegraphics[scale=0.73]{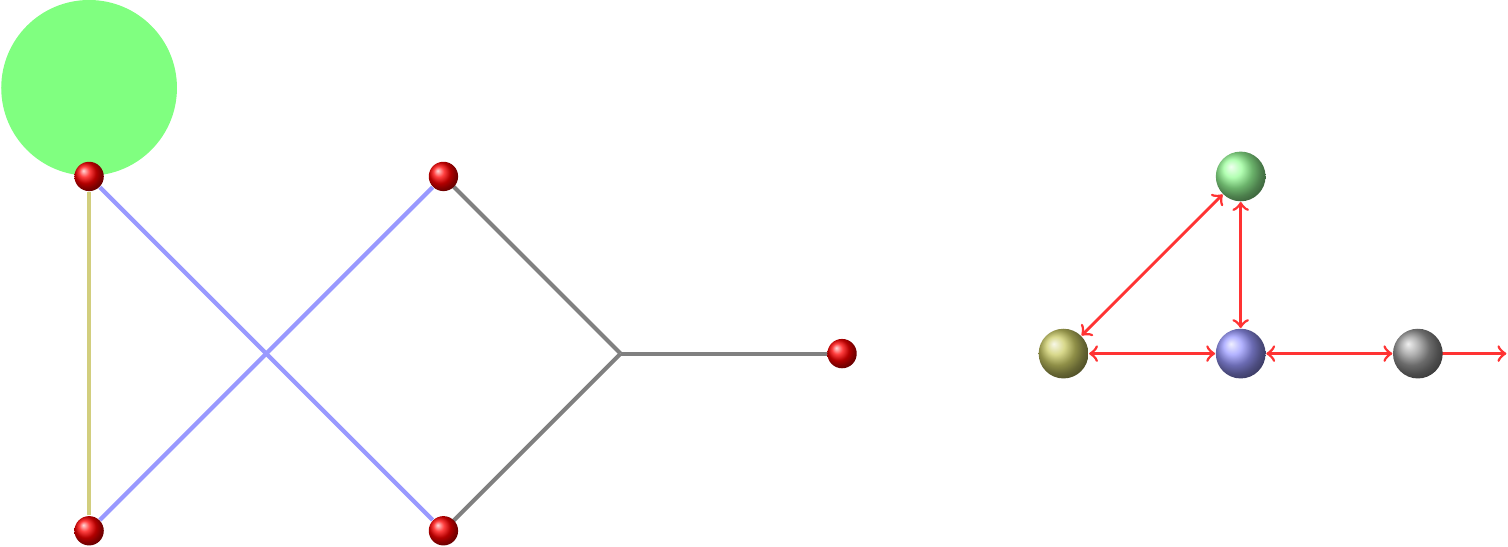}\caption{An averaging principle. (a) On the left: Subregions of the state-space, marked with various colors, are separated by `gates' (in red), such as semi-permeable membranes. Here we have two star-like graphs (perhaps with Skew Brownian motion on them), an interval, and a `general' state-space, like in Example \ref{ex:4}, symbolized by the green circle. (b) On the right: If all constituents of the concatenation are asymptotically splittable, and all these processes are accelerated, but fluxes through the gates remain constant, in the limit we obtain a finite state Markov chain. Each point in the limit state-space is obtained by lumping together all points in one of the subregions of the original state-space.}
\label{gurka}
\end{figure}
\end{center}

\vspace{-0.5cm}

Here are the details. Let $A_{\mathsf v_i}, i \in \mc N$ be the generators of asymptotically splittable processes defined in spaces $\cosi, i \in \mc N.$ Each $A_{\mathsf v_i}$ is characterized by 
\begin{enumerate} 
\item the number of gates $\ka (i) \ge 1$ with the corresponding 
\begin{itemize} 
\item excessive functions $\nad_{\eps}^{i,j}, j =1,\dots, \ka (i)$,
\item exit laws $\ell_{\lam,\eps}^{i,j} = \nad_{\eps}^{i,j}  - \lam \rez{\eps^{-1} A_{\eps \mathsf v_i}}
\nad_{\eps}^{i,j},  j =1,\dots, \ka (i) $,
\item coefficients $\vr^{i,j}$ such that $\sum_{j=1}^{\ka (i)} \vr^{i,j} =1$, describing the probabilities of exiting through gates $j=1,\dots, \ka (i)$,  
\end{itemize}
\item asymptotic exponential lifetime parameter $\gamma_i >0$,
\item projection $P_i$ of $\cosi $ on $C^\flat (S_i)$,
\item the generator $A_{\mathsf 0,i}$ of the corresponding honest Feller process.
\end{enumerate}
As in Section \ref{cbnr} we concatenate these processes with the help of sub-prob\-a\-bil\-i\-ty measures 
\[ \mem_{i,j},  \qquad i,j \in \mc {IJ} \]
on $\suu$ by requiring that a process that exits $S_i$ through the gate $j$ starts anew at a random point of $\suu \setminus S_i$, its distribution at this moment being $\mem_{i,j}$. We stress that in contrast to Section \ref{cbnr}, $\mem_{i,j}$ is assumed to be supported outside of $S_i$: 
\begin{equation}\label{nsap:-1} \mem_{i,j} (S_i) = 0, \qquad j=1,\dots, \ka (i).\end{equation}

Let $\rladue, \lam >0$ be the resolvent of disjoint union of accelerated and tuned processes:
\[ \rladue f \coloneqq \dot{\sum_{i\in \mc N}} \rez{\eps^{-1} A_{\eps \mathsf v_i}} f_i. \]
Also, let (comp. \eqref{concb} with $M=N$)
\begin{equation}\label{nsap:0} \rlace f = \rladue f + \isuma \sum_{j=1}^{\ka(i)} u_{i,j,\eps} \ell_{\lam,\eps}^{i,j}, \qquad f \in \cosu, \end{equation}
where (see Lemma \ref{cnp:lem1}) 
\[ \left (u_{i,j,\eps} \right )_{(i,j)\in \mc {IJ}} = (I - N_{\lam,\eps})^{-1}  \left (\int_S \rladue g \ud \mem_{i,j} \right )_{(i,j)\in \mc {IJ}} \]
for $N_{\lam,\eps}: \R^\ka \to \R^\ka$ ($\ka\coloneqq \sum_{i=1}^N\ka(i)$) given  by 
\begin{equation}\label{nsap:1} N_{\lam,\eps}  \left (w_{i,j} \right )_{(i,j)\in \mc {IJ}} = \left ( {\sum_{k\in\mc N}} \sum_{l=1}^{\ka (k)}w_{k,l} \int_{S_k} \ell_{\lam,\eps}^{k,l} \ud \mem_{i,j} \right )_{(i,j)\in \mc {IJ}}. \end{equation}
This formula defines the resolvent of a concatenated process in which evolution in each $S_i$ is accelerated while the fluxes of probability mass through all the exits are controlled. 

Our main theorem is concerned with the limit behavior of the semigroups related to $\rlace, \lam >0$; we will show that they converge to the semigroup related to a Markov chain in $\mc N$ with intensity matrix $Q=\left (q_{i,k}\right )_{i,k\in \mc N},$ with the following entries:
\begin{equation}\label{mequ} q_{i,i}\coloneqq - \gamma_i,\quad q_{i,k} \coloneqq \gamma_i \sum_{j=1}^{\ka (i)} \vr^{i,j} \mem_{i,j} (S_k), \quad k \not = i, \qquad i \in \mc N.\end{equation}
Here, $\vr^{i,j}$ is the probability that the approximating process started in $S_i$ will exit through the $j$th gate, and $\mem_{i,j}(S_k)$ is the probability that after exiting through this gate it will start anew somewhere in $S_k$.

\begin{thm}\label{thm:ap} Let $\macoe$ be the generator of the semigroup related to $\rlace, \lam >0$ so that 
$\rez{\macoe} = \rlace, \lam >0$. Then
\[ \grae \e^{t \macoe} f= \e^{tQ} Pf , \qquad t>0, g \in \cosu \]  
with the limit uniform for $t$ in compact subsets of $(0,\infty)$. Here, 
$P$ given by 
\[ Pf = \isuma P_i f_i, \qquad f \in \cosu \]
is the projection of $\cosu$ onto the space  $C^\flat (\suu) $ of functions that are constant on each of the spaces $S_i$ separately. The space    $C^\flat (\suu) $ is isometrically isomorphic to $\R^N$ equipped with the maximum norm, and thus each operator in this space can be identified with a matrix. Here the operator $Q: C^\flat (\suu) \to C^\flat (\suu) $ is identified with the matrix $Q$ of \eqref{mequ}.
\end{thm}
\begin{proof} 
Clearly, we are dealing with a particular case of the situation described in Kurtz's Theorem.

\textbf{Step 1. }  \textbf {Limit of $\rlace, \lam >0$, as $\eps \to 0$.} By assumption (points (a) and (d) in Definition \ref{asik}) 
$\grae \ell_{\lam,\eps}^{k,l} = \frac{\vr^{k,l}\gamma_k}{\lam +\gamma_k}1_{S_k}, $ and in particular 
\[ \grae \int_{S_k}\ell_{\lam,\eps}^{k,l} \ud \mem_{i,j} = \frac{\vr^{k,l}\gamma_k}{\lam +\gamma_k} \mem_{i,j} (S_k) \qquad (k,l),(i,j) \in \mc{IJ}.\]
Hence, the operators $N_{\lam,\eps}$ of \eqref{nsap:1} converge as $\eps \to 0$ to $N_{\lam,0}$ defined by 
\begin{equation}\label{nsap:2} N_{\lam,0}  \left (w_{i,j} \right )_{(i,j)\in \mc {IJ}} = \left ( {\sum_{k\in\mc N}} \sum_{l=1}^{\ka(k)}w_{k,l} \frac{\vr^{k,l}\gamma_k}{\lam +\gamma_k} \mem_{i,j} (S_k) \right )_{(i,j)\in \mc {IJ}}. \end{equation}
As in Lemma \ref{cnp:lem1},
\[\|N_{\lam,0} \| \le  \max_{(i,j)\in \mc{IJ}} \sum_{k\in \mc M} \sum_{l=1}^{\ka (k)} \frac{\vr^{k,l}\gamma_k}{\lam +\gamma_k} \mem_{i,j} (S_k)=   \max_{(i,j)\in \mc{IJ}} \sum_{k\in \mc M} \frac{\gamma_k}{\lam +\gamma_k} \mem_{i,j} (S_k)  < 1.\]
At the same time, by point (d) in Definition \ref{asik}, 
\[ \grae \rladue g = \dot{\sum_{k\in \mc N}} \frac 1{\lam +\gamma_k} P_k g_k  = \dot{\sum_{k\in \mc N}} \frac {\widetilde F_kg_k}{\lam +\gamma_k} 1_{S_k},
\]
where $\widetilde F_{k}$ is the functional defining $P_k$: $P_k g_k = (\widetilde F_{k} g_k) 1_{S_k}$. It follows that 
$\left (u_{i,j,\eps} \right )_{(i,j)\in \mc{IJ}}$ converges, as $\eps \to 0$, to 
\[ \left (u_{i,j,0} \right )_{(i,j)\in \mc{IJ}} = (I - N_{\lam,0})^{-1} \left ( v_{i,j,0} \right )_{(i,j)\in \mc{IJ}}, \]
where \[ v_{i,j,0} \coloneqq \sum_{k\in \mc N}\frac{\mem_{i,j}(S_k) \widetilde F_{k}g_k}{\lam + \gamma_k} .\]  
Therefore, by \eqref{nsap:0}, 
\[ \rlac g \coloneqq \grae \rlace g = \isuma \frac {\widetilde F_i g_i}{\lam +\gamma_i} 1_{S_i}  + \isuma \sum_{j=1}^{\ka (i)} u_{i,j,0} \frac{\vr^{i,j}\gamma_i}{\lam +\gamma_i}1_{S_i}. \]

\textbf{Step 2. }\textbf {Identification of $\rlac $.}
The formula obtained above implies that $\rlac g$ is constant on each $S_i,i\in \mc N$ separately, and its value at each $s\in S_i$ is 
\begin{equation} f(i) \coloneqq \frac {\widetilde F_i g_i}{\lam +\gamma_i} + \sum_{j=1}^{\ka (i)} u_{i,j,0} \frac{\vr^{i,j} \gamma_i}{\lam + \gamma_i} \label{nsap:3}. \end{equation}
Using this relation, definition \eqref{mequ} and assumption \eqref{nsap:-1},
we calculate
\begin{align*}
\sum_{k\not = i} q_{i,k} f(k) &= \gamma_i \sum_{k \in \mc N} \frac {\widetilde F_k g_k}{\lam +\gamma_k} \sum_{l=1}^{\ka (i)} \vr^{i,l}p_{i,l} (S_k)\\ &\phantom{=} 
+\gamma_i \sum_{l=1}^{\ka (i)} \vr^{i,l} \left [ 
\sum_{k\in \mc N} \sum_{j=1}^{\ka (k)} u_{k,j} \frac{\vr^{k,j} \gamma_k p_{i,l}(S_k)}{\lam + \gamma_k} \right ],  \end{align*} 
and the first term here is $\gamma_i \sum_{l=1}^{\ka (i)} \vr^{i,l} v_{i,l}$. Since $(u_{i,j,0})_{(i,j)\in \mc {IJ}}$ is a unique solution to the equation $ (u_{i,j,0})_{(i,j)\in \mc {IJ}} = (v_{i,j,0})_{(i,j)\in \mc {IJ}} + N_{\lam,0} (u_{i,j,0})_{(i,j)\in \mc {IJ}}$, the expression in brackets equals $u_{i,l} - v_{i,l}$. Thus, the entire sum reduces to 
$\gamma_i \sum_{l=1}^{\ka (i)} \vr^{i,l} u_{i,l}$. This in turn, by \eqref{nsap:3}, is $(\lam+\gamma_i) f(i) - \widetilde F_i g_i$. This means that 
\[ (\lam - Q) (f(i))_{i \in \mc N} = (\widetilde F_i g_i)_{i\in \mc N}.\]
By identifying $Q$ with an operator in $C(\suu)$ we obtain thus
\[ \rlac g = \rez{Q}Pg .\]

\textbf{Step 3. }\textbf {Convergence of $\rez{\eps \macoe}$, as $\eps \to 0$.} For each $i\in \mc N$, let $A_{\mathsf 0,i}$ be the generator of the honest Feller process of point (d) in Definition \ref{asik}.  By assumption, 
\begin{align}\grae \eps^{-1} R_{\frac \lam \eps, \eps}^\du f &= \grae \isuma \rez{\eps A_{\eps \mathsf v_i}} g_i = \isuma \rez{\eps A_{\mathsf 0,i}} g_i \nonumber \\ &=\rez{B}g, \label{nsap:4} \end{align}
where $B$ is the generator of disjoint union of the processes generated by $A_{\mathsf 0,i}$ in $\cosi, i \in \mc N.$ Again by assumption,
\[ \grat \e^{tB} f = Pf, \qquad f \in \cosu .\]
Therefore, by Kurt'z Theorem, to complete the proof we need to show that $\rez{\eps\macoe} g = \eps^{-1} \left (\frac \lam \eps - \macoe \right )^{-1}g = \eps^{-1} R_{\frac \lam \eps, \eps}^\co g $ 
converges to $\rez{B}g$ as $\eps \to 0$. 

By \eqref{nsap:0}, 
\[ \eps^{-1} R_{\frac \lam \eps, \eps}^\co = \eps^{-1} R_{\frac \lam \eps, \eps}^\du + \isuma \eps^{-1} \sum_{j=1}^{\ka (i)} u_{i,j,\eps} (f, \eps^{-1} \lam ) \ell_{\frac \lam \eps, \eps}^{i,j},\]
where 
\[ \left | \eps^{-1}  u_{i,j,\eps} (f,\eps^{-1} \lam)\right | =   \left | \eps^{-1} \int_S R_{\frac \lam \eps, \eps}^\co f \ud \mem_{i,j} \right |\le \eps^{-1} \| R_{\frac \lam \eps, \eps}^\co f\| \le \lam^{-1} \|f\|. \]
Hence, by \eqref{nsap:4}, we are left with showing that 
\[ \grae \ell_{\frac \lam \eps, \eps} ^{i,j} =0, \qquad j=1,\dots, \ka (i), i\in \mc N, \lam >0.\]
However, 
\[ \ell_{\frac \lam \eps, \eps} ^{i,j} = \nad_\eps^{i,j} - \frac \lam \eps \left (\frac \lam \eps - \frac 1\eps A_{
\eps \mathsf v_i}\right )^{-1} \nad^{i,j}_\eps  =  \nad_\eps^{i,j} - \lam  \left (\lam  -  A_{
\eps \mathsf v_i}\right )^{-1} \nad^{i,j}_\eps   \]
and this converges, by assumption, to $\vr^{i,j}1_{S_i} - \lam \rez{A_{\mathsf, i}}\vr^{i,j} 1_{S_i}=0$, because the processes generated by $A_{\mathsf 0,i}$s are honest. 
\end{proof} 

\section{Appendix}\label{app}

 \subsection{$\lam \mapsto \frac {\cosh \sdlam (r-x)}{\cosh \sdlam r}$  is the Laplace transform of the distribution of the time needed for the Brownian motion starting at $x\in (0,r)$ and reflected at $x=r$ to reach $x=0$ for the first time}\label{ajeden} 
Formula (8.29) p. 100 in \cite{karatzas} says that the Laplace transform of the distribution of the time needed for the Brownian motion starting at an $x>0 $ to reach $0$ or $a>x$ is 
\( \frac {\cosh \sdlam (x-\tfrac a2)}{\cosh \sdlam  \tfrac a2}. \)
On the other hand, point $0$ is reached by the Brownian motion starting at $x$ and reflected at $r$ iff $0$ or $2r$ is reached by the ordinary (not reflected) Brownian motion starting at $x$. Hence, the Laplace transform we are searching for is obtained by replacing $a$ by $2r$ in the formula above.

\subsection{$k_\lam \coloneqq \frac {c \cosh \sdlam r} {c \cosh \sdlam r + \sdlam \sinh \sdlam r}$ is completely monotone}\label{adwa}  
The proof is based on the following expansion of $ \cosh z +  az  \sinh z $, where $a\not= 0$ is a  constant, into an infinite product (take e.g. $d=0$  in formula 13. p. 263 in \cite{goodson}) which can be proved using the Hadamard Factorization Theorem \cite{titchmarsh}. Let $p_n = p_n(a), n \ge 0$ be all positive solutions to the equation $\tan p = \frac 1{ap}$ (arranged in the increasing order). Then
\[ \cosh z + a z \sinh z  = \prod_{n=0}^\infty  \left (1+ \frac {z^2}{p_n^2}\right ), \qquad z \in \mathbb C.  \]
For $z= \sdlam r$ and $a=\frac 1{rc}$, this yields
\[  \frac {c} {c \cosh \sdlam r + \sdlam \sinh \sdlam r} = \prod_{n=0}^\infty \frac {p_n^2}{2r^2\lam +p_n^2} , \qquad \lam >0,\]
where $p_n=p_n(\frac 1{rc}).$ Analogously (see \cite{goodson} p. 263 formula 1.),
\[ \cosh z = \prod_{n=0}^\infty  \left (1+ \frac {z^2}{q_n^2}\right ), \qquad z \in \mathbb C, \]
where $q_n = \tfrac \pi 2 + n\pi .$ Therefore, 
\[ k_\lam = \prod_{n=0}^\infty \left ( \frac {p_n^2}{q_n^2} \frac {q_n^2 + 2\lam r^2}{p_n^2 + 2\lam r^2}\right ) = \prod_{n=0}^\infty  \left [ \frac {p_n^2}{q_n^2} \left (1 + \frac {q_n^2 - p^2_n}{p_n^2 + 2\lam r^2} \right )\right ].\]
Since, by definition, $p_n< q_n$, it follows that each factor in this expansion is completely monotone. 
Recalling that finite products of completely monotone functions are completely monotone, we see that for each $N\in \N$ the product of $N$ first factors of $k_\lam $ is completely monotone. 
Therefore, by the extended continuity
theorem for the Laplace transform (see \cite{feller}*{p. 433}) combined with the Bernstein Theorem, so is $k_\lam$.

\subsection{Formulae for determinants of Section \ref{exa2} and their limit forms}\label{aff} We note the following more explicit formulae for  the main determinant $\wyr$ and determinants $\wyr_1$ and $\wyr_2$ of \eqref{exa2:6}:
\begin{align*}
{\textstyle \frac 12}\wyr &= -[(2\lam q +d) (2\lam p +c) + 2\lam (1-p)(1-q)] \sinh\slam r \\
&\phantom{=} + \slam [(1-p)(2\lam q + d)+ (1-q)(2\lam p +c) ]\cosh \slam r, \\
{\textstyle \frac 12}\wyr_1&= -(2\lam q+d) \sinh \slam (r-x) - \slam (1-q) \cosh \slam (r-x),\\
{\textstyle \frac 12}\wyr_2&= -(2\lam p+c)\sinh\slam x - \slam (1-p) \cosh \slam x, \quad x\in [0,r].
\end{align*} 
As an immediate result, 
 \begin{align} - \lilz \frac \wyr \slam &= 2cdr + 2(1-q) c + 2 (1-p) d , \nonumber \\ 
-\lilz \frac {\wyr_1}{\slam} &= 2d(r-x) + 2(1-q), \nonumber \\
-\lilz \frac {\wyr_2}{\slam} &= 2cx + 2(1-p).\label{aff:1}\end{align}
Moreover, if all occurrences of $\lam, c$ and $d$ in these determinants are replaced by $\eps \lam, \eps c$ and $\eps d$, respectively, 
\begin{align}
\grae \frac{\wyr}{\eps^{\frac 32}} &= - 2 \slam (2\lam M  + c(1-q) + d(1-p)), \nonumber \\
\grae \frac{\wyr_1}{\eps^{\frac 12}} &= - 2 \slam (1-q), \nonumber \\
\grae \frac{\wyr_2}{\eps^{\frac 12}} &= - 2 \slam (1-p),\label{aff:2} \end{align}
where $M $ is defined in \eqref{eoas:1}.

Similarly, 
\[ {\textstyle \frac 12} (a_{2,1} +a_{2,2}) = (2\lam+d)\cosh \slam r + 2\slam (1-q) \sinh \slam r.\] 
It follows that, if $d=0$,
\begin{equation}\label{aff:3} \lilz \frac {a_{2,1} +a_{2,2}}{4\lam} = q +(1-q)r.\end{equation}
Moreover, if $\lam$ and $q$ are replaced by $\eps \lam $ and $\eps q$, respectively, (but $d$ needs not be zero) 
\begin{equation} \label{aff:4} \lilz \frac {a_{2,1} +a_{2,2}}{\eps } = 2\lam q +d +2\lam (1-q)r.\end{equation}
\subsection{Proof of \eqref{exa3:6a}}\label{puff} 
We have
\[ (\ell_\lam^j)'_k (0) = \frac {2d\slam}{a_{2,1}-a_{2,2}} \delta_{j,k} - \frac {4d\lam (a_{2,1} + a_{2,2})}{m_\lam (a_{2,1}-a_{2,2})^2}\alpha_j, \quad j,k\in \kad. \]
Therefore, 
\begin{align*}
\sum_{k\in \kad}\alpha  (\ell_\lam^j)'_k (0) &= \frac{2d\slam}{a_{2,1} - a_{2,2}}\alpha_j - (1-\beta) \frac {4d\lam (a_{2,1} + a_{2,2})}{m_\lam (a_{2,1}-a_{2,2})^2}\alpha_j \\
&= \frac{4d\lam \beta\slam}{m_\lam(a_{2,1}-a_{2,2})}. \end{align*}
Since, on the other hand,
\( \beta f''(0)=2\lam \beta f(0) =  \frac{4d\lam \beta \slam}{m_\lam(a_{2,1}-a_{2,2})},\)
 we are done. 
  
\vspace{0.2cm}
\textbf {Acknowledgment.}  This research is supported by National Science Center (Poland) grant
2017/25/B/ST1/01804.

I would like to thank K. Bogdan for several many-hour-long discussions on the subject of the paper, and for numerous bibliographical items. Moreover, I would like to acknowledge the impact the paper has gained from E.~Ratajczyk, A. Gregosiewicz and Ł. Stepie\'n. In particular, the reasoning of Section \ref{adwa} I owe to E.R. and A.G., and I am grateful to E.R. for correcting my faulty calculations in the original version of Example \ref{ex:5}.


\bibliographystyle{plain}
\bibliography{bibliografia}

\end{document}